\documentclass[10pt]{amsart}
\usepackage{amsmath,amsfonts,amsthm,amsopn,color,amssymb,enumitem}
\usepackage{bbm}
\usepackage{graphicx,tikz,mathtools}
\usepackage[left=2.65cm,right=2.65cm,top=2.7cm,bottom=2.7cm]{geometry}
\usepackage[pdftex]{hyperref}
\usepackage{cite}
\usepackage{mathrsfs} 
\usepackage{caption}
\hypersetup{
	colorlinks=true,
	linkcolor=blue, 
	citecolor=blue,
	filecolor=blue,
	urlcolor=blue,
}

\usepackage{esint}
\usepackage[title]{appendix}

\newtheorem{theorem}{Theorem}[section]
\newtheorem{lemma}[theorem]{Lemma}

\newtheorem{proposition}[theorem]{Proposition}
\newtheorem{remark}[theorem]{Remark}

\newtheorem{corollary}[theorem]{Corollary}

\def\beq{\begin{equation}}
\def\eeq{\end{equation}}
\def\ba{\begin{array}}
\def\ea{\end{array}}

\def\R{\mathbb R}

\def\N{\mathbb N}

\newcommand{\rmnote}[1]{}%{\mnote{#1}}

\numberwithin{equation}{section}

\newenvironment{key words}{\textbf{Keywords}\mbox{  }}{ }

\allowdisplaybreaks

\begin{document}

%---------------------------------------------------------------------------------

%%% ---------------------------------------------------

\title[Short title]{Multi-peak solutions for the fractional Schrödinger equation with Dirichlet datum}

\author[M. Medina]{Maria Medina}
\address{ \vspace{-0.4cm}
\newline 
\textbf{{\small Maria Medina}} 
\vspace{0.15cm}
\newline \indent  
Departamento de Matem\'aticas, Universidad Aut\'onoma de Madrid, 28049 Madrid, Spain}
 \email{maria.medina@uam.es}
 
\author[J. Wu]{Jing Wu}
\address{ \vspace{-0.4cm}
\newline 
\textbf{{\small Jing Wu}} 
\vspace{0.15cm}
\newline \indent  
Departamento de Matem\'aticas, Universidad Aut\'onoma de Madrid, 28049 Madrid, Spain}
 \email{jing.wu@uam.es}

\keywords{Fractional Laplacian, Schr\"odinger equation, concentration phenomena, Lyapunov-Schmidt reduction}

\subjclass[2020]{35R11}

%%    General info
%\subjclass[2010]{35B38, 58J05, 58K45}
%\date{\today}
%
%\keywords{ }
%
\begin{abstract}  \vspace{-0.2cm}
Let $s\in (0,1)$, $\varepsilon>0$ and let $\Omega$ be a bounded smooth domain. Given the problem
$$\varepsilon^{2s}(-\Delta)^{s} u + V(x)u = |u|^{p-1}u \quad \mbox{in }\; \Omega,$$
with Dirichlet boundary conditions and $1<p<(n+2s)/(n-2s)$, we analyze the existence of positive multi-peak solutions concentrating, as $\varepsilon\to 0$, to one or several points of $\Omega$. Under suitable conditions on $V$, we construct positive solutions concentrating at any prescribed set of its non degenerate critical points. Furthermore, we prove existence and non existence of clustering phenomena around local maxima and minima of $V$, respectively. The proofs rely on a Lyapunov-Schmidt reduction where three effects need to be controlled: the potential, the boundary and the interaction among peaks. The slow decay of the associated {\it ground-state} demands very precise asymptotic expansions.
\end{abstract}

\maketitle

%%% ---------------------------------------------------

\section{Introduction}
\label{Section 1}
Let $\Omega$ be a bounded smooth domain of $\R^n$, $n\geq 2$, and $s\in (0,1)$. Given a small parameter $\varepsilon>0$, we consider the semilinear fractional problem
\begin{equation}\label{dirichlet}
\left\{\begin{array} {ll}
\varepsilon^{2s}(-\Delta)^{s} u + V(x)u = |u|^{p-1}u & \mbox{in }\; \Omega,\\
 u= 0 & \mbox{in }\; \mathbb{R}^{n}\setminus\Omega,
\end{array}\right.
\end{equation}
where $p \in\big(1, \frac{n+2s}{n-2s}\big)$, and $V\in C^{2}(\overline{\Omega})$ is a positive potential satisfying 
\begin{equation}\label{eq:infBoundV}
\inf\limits_{x\in\Omega}V(x)>0.
\end{equation} 
Here $(-\Delta)^{s}$ stands for the fractional Laplacian, defined as
\[(-\Delta)^{s} u(x):=c_{n,s}\lim_{\delta\to 0}\int_{\R^n\setminus B_\delta(x)}\frac{u(x)-u(y)}{|x-y|^{n+2s}}\,dy,\qquad x\in \R^n,\]
where $c_{n,s}$ is a suitable normalization constant (see for instance \cite{BV2016, L1972, S2007} for interesting motivations and properties of this operator). The goal of this paper is to study a series of concentration phenomena for solutions of \eqref{dirichlet} in the form of \textit{multi-peak} or \textit{multi-bump} solutions.

Concentration phenomena for the classical semilinear Schr\"odinger equation have been largely studied due to the great variety of possible behaviors. When the problem is posed in the whole space,
$$-\varepsilon^2\Delta u +V(x)u=|u|^{p-1}u\qquad \mbox{ in }\R^n,\qquad p\in \big(1, \tfrac{n+2}{n-2}\big),$$
spiked positive solutions where constructed first by Floer and Weinstein and by Oh in \cite{FW1986, O1989, O1990}. In these works the peaks of the solutions are well separated one from each other and their locations converge to non degenerate critical points of $V$ as $\varepsilon\to 0$. A clustering phenomenon (i.e., solutions with several interacting peaks concentrating at the same point) around local maxima of $V$ was proven by Kang and Wei in \cite{KW2000}, where the authors also showed that this behavior cannot happen around local minima. Clustering at this type of critical points have been shown in the case of sign-changing solutions, where mixed positive and negative spikes are considered (see for instance \cite{AS2004, BCW2007, AP2007, AP2009, dPFT2002}). In \cite{AR2011} an analogous behavior is found around saddle points of the potential. All these results are proven by perturbation arguments, using as concentration profile the {\it ground-state} solution of 
$$-\Delta w+w=w^p,\quad H^1(\R^n),$$
which is positive, radial, and decays exponentially. The situation importantly changes if the problem is posed in a bounded domain $\Omega$. When we neglect the effect of the potential (i.e., we fix $V\equiv 1$), the boundary conditions play the main role in the concentration phenomenon. In the case of Dirichlet conditions, Ni and Wei proved in \cite{NW1995} that concentration may only happen at a single interior point of the domain, which corresponds to a maximizer of the distance to the boundary. When Neumann boundary conditions are imposed, the solution also has a unique local maximum, but located at the boundary, at a precise point that maximizes the mean curvature of $\partial \Omega$ (see \cite{NT1991,NT1993}). In the presence of a suitable positive potential, by means of variational techniques del Pino and Felmer proved existence of solutions concentrating at any prescribed set of local minima, possible degenerate, of the potential (see \cite{dPF1998}).

In the fractional framework, an essential difference arises. If one wants to apply perturbative methods (as it is the case of this manuscript), it is natural to consider as concentration profile the unique radial positive least energy solution (see \cite{FQT2012,FLS2016}) of
\begin{equation}\label{gsF}
(-\Delta)^{s} w + w = w^{p},\quad w\in H^{s}(\mathbb{R}^{n}).
\end{equation}
This function is smooth and it has the asymptotic behavior 
\begin{equation}\label{behavw}
\frac{\alpha}{1+|x|^{n+2s}}\leq w(x)\leq\frac{\beta}{1+|x|^{n+2s}},
\end{equation}
for certain positive constants $\alpha,\beta$. The exponential decay of the local case is here replaced by a polynomial-type rate, which makes substantially stronger the superposition of the tails of the different peaks (i.e., the interaction among them) and the error of the boundary data. In the case of $\R^n$, where only the first difficulty appears, after much more involved computations than in the local case, D\'avila, del Pino and Wei proved in \cite{DdPW2014} the non local counterpart of the results by Floer, Weinstein and Oh, that is, concentration phenomena at non degenerate critical points of the potential. In \cite{ARS2021}, Alarc\'on, Ritorto and Silva studied the corresponding clustering phenomena.

In the case of bounded domains major difficulties emerge. In the beautiful article \cite{DdPDV2015}, D\'avila, del Pino, Dipierro and Valdinoci considered problem \eqref{dirichlet} when $V\equiv 1$, and they aimed to extend \cite{NW1995} to the fractional framework. After very delicate energy expansions, they were able to prove that concentration may occur at a single point, located in the interior of the domain. However, the great effect provoked by the slow decay of the profile prevented them from identifying the precise location of the concentration point. Up to our knowledge, this remains as an open problem. It is then natural to consider the case $V\not\equiv 1$ and to analyze the existence or non existence of multi-peak solutions to \eqref{dirichlet}, and the precise location of the corresponding concentration points. This is  the goal of this paper. In order to be more precise, let us point out that, if $w$ solves \eqref{gsF}, for every parameter $\lambda>0$ the rescaled function
\begin{equation}\label{eq:w_lambda}
w_\lambda(x):=\lambda^{\frac{1}{p-1}}w(\lambda^{\frac{1}{2s}}x)
\end{equation}
satisfies
\begin{equation}\label{eqwlambda}
(-\Delta)^{s} w_\lambda + \lambda w_\lambda = w_\lambda^{p} \quad \mbox{in} \, \,\mathbb{R}^{n}.
\end{equation}
Hence, for any point $\xi\in\Omega$, the spike-shaped function
$w_{V(\xi)}\big(\frac{x-\xi}{\varepsilon}\big)$
solves the problem
\begin{equation}\label{gsP}
\varepsilon^{2s}(-\Delta)^{s} v + V(\xi) v = v^{p} \quad \mbox{in} \; \mathbb{R}^{n}.
\end{equation}
Notice that, as $\varepsilon\to 0$, this function exhibits a concentration phenomenon at the point $\xi$. Actually, we will construct multi-peak solutions to \eqref{dirichlet} whose concentration profiles at every peak are asymptotically this. Three main effects/errors coexist in this setting (with the polynomial decay of $w$ playing a crucial role):
\begin{itemize}
\item The potential: problem \eqref{gsP} assumes the potential to be constant, which differs from the situation in \eqref{dirichlet}. 

\item The interaction among peaks: given the non linear nature of the problem, the superposition of profiles concentrating at different points will produce an inevitable error.

\item The boundary correction: the fact that the profile is a strictly positive function in the whole space makes necessary to introduce a boundary correction in order to satisfy the boundary condition of \eqref{dirichlet}. This requires an involved analysis of the asymptotics of the associated Green and Robin functions.
\end{itemize}
The first two items are also present in \cite{ARS2021, DdPW2014} and determine the fact that concentration happens at critical points of $V$. The third appears in \cite{DdPDV2015} and it is the reason why in their case the peak is located at an (unspecified) interior point. In this manuscript we will deal with the three effects at the same time. By using the variational structure of the problem, we will obtain sharp asymptotic expansions of the energy that will allow us to identify the item which produces the biggest effect and therefore determines the location of the concentration phenomena. Actually, in the first result we will see that, given a prescribed set of different non degenerate critical points of $V$, the interaction among peaks centered at them (via the potential) hides the effect of the boundary, which appears at a lower order in $\varepsilon$.

\begin{theorem}\label{result1} Let $k\in \N$, $k\geq 1$, and let $\;\hat{\xi}_1,\ldots,\hat{\xi}_k\in \Omega$ be $k$ non degenerate critical points of $V$, i.e., such that 
$$\nabla V(\hat{\xi}_i)=0\quad \mbox{ and }\quad  det\,(D^2V(\hat{\xi}_i))\neq 0, \quad \mbox{ for every }i\in\{1,\ldots,k\}.$$
Then, there exists $\varepsilon_\star>0$ such that, for every $0<\varepsilon<\varepsilon_\star$, 
\begin{equation}\label{eq:u_eps}
u_\varepsilon(x)=\sum\limits_{i=1}^k w_{V(\xi_i^\varepsilon)}\bigg(\frac{x-\xi_i^\varepsilon}{\varepsilon}\bigg)+o_\varepsilon(1),
\end{equation}
is a $k$-spike solution of \eqref{dirichlet}, with $\xi^\varepsilon_i\to \hat{\xi}_i$ as $\varepsilon\to 0$.
\end{theorem}

Here $o_\varepsilon(1)$ stands for a quantity that vanishes uniformly as $\varepsilon\to 0$. More precise information will be actually given: this term corresponds to a small function whose $\|\cdot\|_*$-norm  (see \eqref{rho} in Section \ref{Section 5}) decays with $\varepsilon$. Notice that, using a different technique, this theorem extends to the non local case the result in \cite{dPF1998}, but allowing all type of critical points of $V$, not only minima (although requiring them to be non degenerate). 

The analysis performed to prove Theorem \ref{result1} involves the identification of the exact order at which the interactions among bumps occur, compared to the other effects. It can be deduced from it that there is some room to consider configurations where the points are different but collapsing when $\varepsilon\to 0$, as long as this convergence is not faster than a certain velocity. This allows us to conclude the existence of clustering phenomena around local maxima of the potential, and the non existence around local minima. 
The case of sign-changing multi-peak solutions is left as an open problem.

\begin{theorem}\label{result3} Let $k\in \N$, $k\geq 1$, and let $K$ be a bounded open set of $\Omega$ with smooth boundary such that
\[\sup\limits_{K}V>\sup\limits_{\partial K}V.\]
Then, for every $\alpha\in (0,1)$ there exists a $k$-spike solution $u_\varepsilon$ of \eqref{dirichlet} of the form \eqref{eq:u_eps} with $\xi_i^\varepsilon\in K$ and 
\begin{equation}\label{distance}
|\xi_i^\varepsilon-\xi_\ell^\varepsilon|>\varepsilon^{1-\frac{\alpha}{n+2s}},\;\; i\neq \ell,\qquad i,\ell\in\{1,\ldots,k\}.
\end{equation}
Furthermore, if $\hat{\xi}$ is a strict local maximum point of $V$ in $\Omega$,  there exists a $k$-spike solution of the form \eqref{eq:u_eps} satisfying \eqref{distance} and such that
$$\xi^\varepsilon_i\to \hat{\xi}\;\;\mbox{ as }\;\;\varepsilon\to 0.$$
\end{theorem}

\begin{theorem}\label{result4} Fix any positive integer $k>1$ and assume that $\hat{\xi}$ is a local minimum point of $V$ such that $det(D^2 V(\hat{\xi}))\neq 0.$ Then equation \eqref{dirichlet} cannot have a positive solution $u_\varepsilon$ of the form \eqref{eq:u_eps} with $\xi_i^\varepsilon$ satisfying
\[\xi_i^\varepsilon\rightarrow \hat{\xi}\qquad \mbox{ and }\quad \frac{|\xi^\varepsilon_i-\xi^\varepsilon_\ell|}{\varepsilon}\to+\infty\quad  \qquad\emph{as} \,\, \varepsilon\rightarrow 0,\quad i,\ell \in\{1,\ldots,k\}.\]
\end{theorem}
\begin{remark}
It is an interesting open question whether the ideas from Theorem \ref{result1} and Theorem \ref{result3} can be combined to construct multi-peak solutions showing clustering phenomena at certain points, and concentration at distinct points in others, possibly at different rates of convergence. 
\end{remark}
Let us briefly sketch the strategy of the proofs. After absorbing $\varepsilon$ by scaling, equation \eqref{dirichlet} can be rewritten as
\begin{equation}\label{dirichlet1}
\left\{\begin{array} {ll}
(-\Delta)^{s} u + V(\varepsilon x)u = u^{p} & \mbox{in }\; \Omega_\varepsilon,\\
 u= 0 & \mbox{in }\; \mathbb{R}^{n}\setminus\Omega_\varepsilon,
\end{array}\right.\qquad \mbox{ where }\qquad \Omega_\varepsilon:=\frac{\Omega}{\varepsilon}=\left\{\frac{x}{\varepsilon},\;x\in\Omega\right\},
\end{equation}
which is the Euler-Lagrange equation associated to the energy functional
\begin{equation*}
  J_\varepsilon(u):=\frac{1}{2}\int_{\Omega_\varepsilon}((-\Delta)^suu+V(\varepsilon x)u^2)\,dx-\frac{1}{p+1}\int_{\Omega_\varepsilon} u^{p+1}\,dx, \quad u\in H^s_0(\Omega_\varepsilon).
\end{equation*}
Let $k\in\N$, $k\geq 1$, and consider $\xi_1,\ldots,\xi_k\in\Omega$ and their rescaled versions 
\begin{equation}\label{xi_i}
q_i:=\frac{\xi_i}{\varepsilon}\in\Omega_\varepsilon,\quad i\in\{1,\ldots, k\}.
\end{equation}
and denote 
\begin{equation}\label{wi}
w_i(x):=w_{\lambda_i}(x-q_i), \quad \lambda_i=V(\xi_i),\quad x\in\Omega_\varepsilon,
\end{equation}
with $w_{\lambda_i}$ defined in \eqref{eq:w_lambda}. By $\eqref{behavw}$,
\begin{equation*}
  \lambda_i^{\frac{1}{p-1}-\frac{n+2s}{2s}}\frac{\alpha}{|x-q_i|^{n+2s}}\leq w_i(x)\leq \lambda_i^{\frac{1}{p-1}-\frac{n+2s}{2s}}\frac{\beta}{|x-q_i|^{n+2s}}, \quad\mbox{ for }|x-q_i|\geq \lambda_i^{-\frac{1}{2s}}.
\end{equation*}
\begin{remark}\label{remark2}Since $\xi_1,\ldots,\xi_k\in\Omega$, without loss of generality we can assume $\xi_1,\ldots,\xi_k\in\Omega^{\delta_*}$, where $\Omega^{\delta_*}:=\{x\in\Omega:\;\mbox{dist}(x,\partial\Omega)> \delta_*\}$ for some $\delta_*\in (0,1)$ fixed. Then $V$ is uniformly bounded from below (and from above due to its regularity) in $\overline{\Omega^{\delta_*}}$ and  there exists $\eta_0>0$ such that $V^{-\frac{1}{2s}}\geq \eta_0$ uniformly in $\overline{\Omega^{\delta_*}}$. Thus, up to renaming the constants $\alpha$ and $\beta$, 
\begin{equation}\label{behavwi}
  \frac{\alpha}{|x-q_i|^{n+2s}}\leq w_i(x)\leq \frac{\beta}{|x-q_i|^{n+2s}} \qquad \mbox{ for } \quad\, |x-q_i|\geq \eta_0,\quad i=1,\ldots,k.
\end{equation}
By simplicity of notation we will assume $\eta_0=1$.
\end{remark}

Since the function $w_i$ does not satisfy Dirichlet boundary conditions, instead of it we will consider the correction $\bar{u}_i$ given as the solution of the linear problem 
\begin{equation}\label{dirichlet2}
\left\{\begin{array} {ll}
(-\Delta)^{s} \bar{u}_i + V(\xi_i)\bar{u}_i = w_i^{p} & \mbox{in }\; \Omega_\varepsilon,\\
 \bar{u}_i= 0 & \mbox{in }\; \mathbb{R}^{n}\setminus\Omega_\varepsilon.
\end{array}\right.
\end{equation}
We will therefore look for our solution as a small perturbation of the superposition of different copies of these functions, that is,
$$u_\varepsilon(x)=\sum_{i=1}^k \bar{u}_i\bigg(\frac{x}{\varepsilon}-q_i\bigg)+\phi(x),$$
with $\phi$ small, which turns out to be of the form \eqref{eq:u_eps} (see Lemma \ref{limit}). This will be done by a Lyapunov-Schmidt type argument: using the information on $\bar{u}_i$ we will transform problem \eqref{dirichlet1} into a non linear problem for $\phi$, which we will solve by using fixed points arguments in an appropriate projected version. After estimating the error of the approximation in a very precise way, an accurate expansion of the energy $J_\varepsilon$ will allow us to reduce the existence of $\phi$ to the solvability of a finite dimensional system. We will finally solve by adjusting the location of the points $q_i$. 

The article is structured as follows: in Section 2 we establish precise asymptotics on the ground state and the Green and Robin functions associated to the problem \eqref{dirichlet2}. Section 3 and 4 are devoted to developing an appropriate solvability theory to find the perturbation $\phi$; first solving the linear version of the problem and then the projected non linear version. Section 5 deals with the delicate energy expansions and the subsequent variational reduction. Section 6 contains the proof of the theorems.

In the sequel $C$ will stand for a positive constant that may change from line to line.

\section{Preliminary results: Green and Robin functions and decay of the ground state}

Let $\Gamma$ be the unique decaying fundamental solution to the problem
\begin{equation}\label{Gammap}
(-\Delta)^{s} \Gamma + \Gamma = \delta_0,
\end{equation}
which satisfies
\begin{equation}\label{behavGamma}
\int_{\R^n}\Gamma(z)\,dz=1\quad \mbox{ and }\quad   \alpha|x|^{-(n+2s)}\leq \Gamma(x)\leq\beta|x|^{-(n+2s)}\qquad \mbox{when }\quad \; |x|\geq1,
\end{equation}
for some positive constants $\alpha,\beta$, see \cite[Lemma C.1]{FLS2016}. 
Fix $\delta_*\in (0,1)$ and define
\begin{equation}\label{Omegaepsstar}
\Omega^{\delta_*}_\varepsilon:=\bigg\{x\in\Omega_\varepsilon:\;\mbox{dist}(x,\partial\Omega_\varepsilon)> \frac{\delta_*}{\varepsilon}\bigg\}.
\end{equation}
Let us consider $q_i,\ldots,q_k$ points of $\Omega^{\delta_*}_\varepsilon$ (or equivalently $\xi_i:=\varepsilon q_i\in \Omega^{\delta_*}$), and denote $\lambda_i:=V(\xi_i)$. 
It can be straightforwardly checked that
\[\Gamma_{\lambda_i}(x):=\frac{1}{\lambda_i}\Gamma(\lambda_i^{\frac{1}{2s}}x)\]
solves the problem
\begin{equation}\label{Gammalambda}
(-\Delta)^{s} \Gamma_{\lambda_i} + \lambda_i \Gamma_{\lambda_i} = \delta_0 \quad \mbox{ in }\quad \mathbb{R}^{n}.
\end{equation}
Furthermore, renaming $\alpha$ and $\beta$, by \eqref{behavGamma} and Remark \ref{remark2} we deduce that 
\begin{equation}\label{behavGammai}
\frac{\alpha}{|x|^{n+2s}}\leq\Gamma_{\lambda_i}(x)\leq \frac{\beta}{|x|^{n+2s}} \quad \mbox{for}\quad |x|\geq 1,\quad i=1,\ldots, k.
\end{equation}
Let $H_{\lambda_i}(\cdot,y)$ be the regular solution to the problem
\begin{equation}\label{Hp}
\left\{\begin{array} {ll}
(-\Delta)^{s} H_{\lambda_i}(\cdot,y) + \lambda_i H_{\lambda_i}(\cdot,y) =0 & \mbox{in }\; \Omega_\varepsilon,\\
 H_{\lambda_i}(\cdot,y) = \Gamma_{\lambda_i} (\cdot-y) & \mbox{in }\; \mathbb{R}^{n}\setminus\Omega_\varepsilon,
\end{array}\right.
\end{equation}
for $y\in \Omega_\varepsilon$. Then
\[G_{\lambda_i}(x,y):=\Gamma_{\lambda_i} (x-y)-H_{\lambda_i}(x,y)\]
is the Green function for $(-\Delta)^s+{\lambda_i}$ in $\Omega_\varepsilon$, that is, $G_{\lambda_i}(\cdot,y)$ solves
\begin{equation}\label{Greenp}
\left\{\begin{array} {ll}
(-\Delta)^{s} G_{\lambda_i}(\cdot,y) + \lambda_i G_{\lambda_i}(\cdot,y) = \delta_y & \mbox{in }\; \Omega_\varepsilon,\\
 G_{\lambda_i}(\cdot,y) = 0 & \mbox{in }\; \mathbb{R}^{n}\setminus\Omega_\varepsilon.
\end{array}\right.
\end{equation}
Proceeding like in \cite[Section 2]{DdPDV2015}, we can establish the following:
\begin{proposition}\label{prop:boundH}
Let $q\in\Omega_\varepsilon$ such that $\rho:=\mbox{dist}(q,\partial\Omega_\varepsilon)\geq 2$. Then, for every $i=1,\ldots, k$,
$$H_{\lambda_i}(x,y)\leq \frac{C}{\rho^{n+4s}},\quad x,y\in B_{\rho/2}(q), $$
for a suitable constant $C>0$ uniform in $q_i\in\Omega^{\delta_*}_\varepsilon$.
\end{proposition}
As a consequence, defining
 %\mathcal{H}_\varepsilon(q_i)&:=\int_{\Omega_\varepsilon}\int_{\Omega_\varepsilon}H_{\lambda_i}(x,y)w_i^p((x)w_i^p(y)dxdy,\\
\[\Pi_i(x):=\int_{\Omega_\varepsilon}w_i^p(y)H_{\lambda_i}(x,y)dy,\quad x\in\R^n,\quad i=1,\ldots, k,\]
and 
\begin{equation}\label{defD}
 d:=\min\{\mbox{dist}(q_i,\partial\Omega_\varepsilon),i=1,\ldots,k\}\geq\frac{\delta_*}{\varepsilon},
\end{equation}
we get:
\begin{proposition}\label{prop:boundPi}
Assume $d\geq 2$. Then, for every $i=1,\ldots, k$,
$$\Pi_i(x)\leq \frac{C}{d^{n+4s}},\quad x\in B_{d/8}(q_i), $$
for a suitable constant $C>0$ uniform in $q_i\in\Omega^{\delta_*}_\varepsilon$.
\end{proposition}
Likewise, considering
$$\Lambda_i(x):=\int_{\R^n\setminus \Omega_\varepsilon}w^p_i(y)\Gamma_{\lambda_i}(x-y)\,dy,\quad x\in\R^n,\quad i=1,\ldots, k,$$
we obtain: 
\begin{proposition}\label{prop:boundLambda}
Assume $d\geq 1$. Then, for every $i=1,\ldots, k$,
$$0\leq \Lambda_i(x)\leq \frac{C}{d^{(n+2s)p}},\quad x\in \Omega_\varepsilon,$$
for a suitable constant $C>0$ uniform in $q_i\in\Omega^{\delta_*}_\varepsilon$.
\end{proposition}
These results can be obtained, via direct adaptations, exactly like Proposition 2.4, Lemma 2.5 and Lemma 3.1 of \cite{DdPDV2015} respectively, so we skip the proof. The only subtle point is the fact that the parameters $\lambda_i$ are uniformly bounded on $i$ from above and from below (see Remark \ref{remark2}), and this allows us to get constants independent of $i$.  In the same way, proceeding like in \cite[Lemma 3.2]{DdPDV2015}, and applying \eqref{behavwi} and \eqref{behavGammai}, we have the following result:
 \begin{lemma}\label{expressionui}Assume $d\geq1$. Then, for every $i=1,\ldots, k$ and every $x\in\Omega_\varepsilon$,
  \[\bar{u}_{i}(x)=w_i(x)-\Lambda_i(x)-\Pi_i(x),\]
  and
  \begin{equation}\label{estexpressui}
  0\leq \Lambda_i(x)=w_i(x)-\bar{u}_{i}(x)-\Pi_i(x)\leq\frac{C}{d^{p(n+2s)}},
  \end{equation}
  for a constant $C>0$ uniform in $q_i\in\Omega^{\delta_*}_\varepsilon$, with $\bar{u}_i$ given in \eqref{dirichlet2}.
\end{lemma}
Notice that, if we set
\begin{equation}\label{Wq}
W_q(x):=\sum\limits_{i=1}^k w_i(x),\quad\mbox{ and }\quad U_q(x):=\sum\limits_{i=1}^k\bar{u}_i(x),\quad x\in \R^n,
\end{equation}
Lemma \ref{expressionui} immediately implies 
\begin{equation}\label{expressionUq}
U_q(x)=W_q(x)-\sum\limits_{i=1}^{k}(\Lambda_i(x)+\Pi_i(x)),\quad x\in\R^n.
\end{equation}
As a consequence of the previous results, we can establish some bounds concerning the interaction among different peaks.

\begin{lemma}\label{minimum2} Assume $d\geq 2$. 
There exists a constant $C>0$, uniform in $q_i\in\Omega^{\delta_*}_\varepsilon$, such that, if $r, t\geq1,$ then
\[\int_{\Omega_\varepsilon} w_i^{r}(x)\Pi_\ell^t(x)\,dx\leq \frac{C}{d^{t(n+2s)}}, \quad i,\ell=1,\ldots,k.\]
\end{lemma}
\begin{proof}
Following the ideas in \cite[Lemma 3.4]{DdPDV2015}, we split the integral into two regions, 
\begin{align*}
  &\int_{\Omega_\varepsilon} w_i^{r}(x)\Pi_\ell^t(x)\,dx=\int_{\{|x-q_\ell|\leq \frac{d}{8} \}} w_i^{r}(x)\Pi_\ell^t(x)\,dx+\int_{\{|x-q_\ell|>\frac{d}{8}\}} w_i^{r}(x)\Pi_\ell^t(x)\,dx=:I+II.
\end{align*}
Using Proposition \ref{prop:boundPi} and the decay properties of $w_i$ (see \eqref{behavwi}), it easily follows that
\begin{align*}
  I&\leq \frac{C}{d^{t(n+4s)}} \int_{\{|x-q_\ell|\leq \frac{d}{8} \}} w_i^{r}(x)\,dx\leq \frac{C}{d^{t(n+4s)}}.
\end{align*}
To estimate $II,$ using Lemma \ref{expressionui} and \eqref{behavwi} we get
\[\Pi_\ell(x)\leq w_\ell(x)\leq\frac{C}{|x-q_\ell|^{n+2s}}\leq \frac{C}{d^{n+2s}},\quad \mbox{ for }|x-q_\ell|>\frac{d}{8}.\]
Therefore, since $r\geq 1$,
\begin{align*}
  II&\leq \frac{C}{d^{t(n+2s)}}\int_{\{|x-q_\ell|> \frac{d}{8} \}} w_i^{r}(x)\,dx\leq\frac{C}{d^{t(n+2s)}},
\end{align*}
and the result follows.
\end{proof}
The previous results are crucial to control the corrections introduced to satisfy the boundary conditions. In the same spirit, the energy expansions will rely on the decay properties of the ground state and its derivatives. Consider $w_i$ introduced in \eqref{wi}, and  define 
\begin{equation}\label{deZij}
  Z_{ij}:=\frac{\partial w_i}{\partial x_j},\quad i=1,\ldots,k,\; j=1,\ldots,n.
\end{equation}
\begin{lemma} \label{Zij} There exists a positive constant $C$, uniform in $q_i\in\Omega^{\delta_*}_\varepsilon$, such that
\[|Z_{ij}(x)|\leq \frac{C}{(1+|x-q_i|)^{\nu_1}}\qquad \mbox{ for any }i=1,\ldots,k,\; j=1,\ldots,n,\]
where $\nu_1:=\min\{(n+2s+1),p(n+2s)\}$.
\end{lemma}
\begin{proof}
By $(5.5)$ in \cite[Lemma 5.2]{DdPDV2015}, we have that
\begin{align*}
|Z_{ij}(x)|=\left|\frac{\partial w_{\lambda_i}(x-q_i)}{\partial x_j}\right|&=\lambda_i^{\frac{1}{p-1}+\frac{1}{2s}}\left|\frac{\partial w(\lambda_i^{\frac{1}{2s}}(x-q_i))}{\partial x_j}\right|\leq \lambda_i^{\frac{1}{p-1}+\frac{1-\nu_1}{2s}}C(1+|x-q_i|)^{-\nu_1}.
\end{align*}
By Remark \ref{remark2} we conclude the result.
\end{proof}
Likewise, applying Remark \ref{remark2} to \cite[Lemma 5.3]{DdPDV2015} we get the following result:
\begin{lemma} \label{graZij} There exists a positive constant $C$, uniform in $q_i\in\Omega^{\delta_*}_\varepsilon$, such that, 
\[|\nabla Z_{ij}|\leq C|x-q_i|^{-\nu_2}\qquad \mbox{for any} \quad\,\, |x-q_i|\geq1,\]
for any $i=1,\ldots,k$, $j=1,\ldots,n,$ where 
$\nu_2:=\min\{(n+2s+2),p(n+2s)\}.$
\end{lemma}
%\begin{proof}
%By the Lemma 5.3 in \cite{DdPDV}, we have that
%\begin{align*}
%|\nabla Z_{ij}|=\left|\nabla \frac{\partial w_{\lambda_i}(x-q_i)}{\partial x_j}\right|=\lambda_i^{(\frac{1}{p-1}+\frac{1}{2s})}\left|\nabla\frac{\partial w(\lambda_i^{\frac{1}{2s}}(x-q_i))}{\partial x_j}\right|\leq \lambda_i^{(\frac{1}{p-1}+\frac{1-\nu_2}{2s})}\tilde{C}|x-q_i|^{-\nu_2}.
%\end{align*}
%Letting $C=\lambda_i^{(\frac{1}{p-1}+\frac{1-\nu_2}{2s})}\tilde{C},$ then the proof is done.
%\end{proof}
The following lemmata establish the orthogonality relations among the functions $Z_{ij}$. Indeed, define
\begin{equation}\label{defeta}
\eta:=\min\{|q_i-q_\ell|:i\neq \ell,\;i,\ell=1,\ldots, k\}\gg 1.
\end{equation}

\begin{lemma} \label{orthogonal} The functions $Z_{ij}$ satisfy, for every $i,\ell=1,\ldots, k$, and every $j,m=1,\ldots,n$, the following condition
\[\int_{\mathbb{R}^{n}}Z_{ij}Z_{\ell m}\,dx=\alpha_i\delta_{i\ell}\delta_{jm}+O(\eta^{-\nu_1}),\qquad \alpha_i:=\int_{\mathbb{R}^{n}}Z^{2}_{i 1}\,dx,\]
where  $\nu_1$ is given in Lemma \ref{Zij}.
\end{lemma}

\begin{proof}
Notice first that, using the radiality of $w$, 
\begin{equation}\label{eq:radialZij}
Z_{ij}(x)=\frac{\partial w_{\lambda_i}(x-q_i)}{\partial x_j}=w'_{\lambda_i}(|x-q_i|)\frac{x_j-(q_{i})_j}{|x-q_i|}.
\end{equation}
Thus, calling $y=x-q_i,$ we have

\[\int_{\mathbb{R}^{n}}Z_{ij}Z_{\ell m}\,dx=\int_{\mathbb{R}^{n}}w'_{\lambda_i}(|y|)w'_{\lambda_\ell}(|y+q_i-q_\ell|)\frac{y_j}{|y|}\frac{y_m+(q_{i})_m-(q_{\ell})_m}{|y+q_i-q_\ell|}\,dy.\]
Hence, if $i=\ell$ and $j=m$,
\begin{equation}\label{eq:iljm1}
\int_{\mathbb{R}^{n}}Z_{ij}^2\,dx=\int_{\mathbb{R}^{n}}w'_{\lambda_i}(|y|)^{2}\frac{y_j^{2}}{|y|^2}\,dy=\int_{\mathbb{R}^{n}}w'_{\lambda_i}(|y|)^{2}\frac{y_1^{2}}{|y|^2}\,dy=\alpha_i,
\end{equation}
and if $i=\ell$ and $j\neq m$,
\begin{equation}\label{eq:iljm2}
\int_{\mathbb{R}^{n}}Z_{ij}Z_{im}\,dx=\int_{\mathbb{R}^{n}}w'_{\lambda_i}(|y|)^{2}\frac{y_j}{|y|}\frac{y_m}{|y|}\, dy=0.\end{equation}
Assume now $i\neq \ell$. Applying Lemma \ref{Zij} we can bound the integral like
\begin{align*}
\left|\int_{\mathbb{R}^{n}}Z_{ij}Z_{\ell m}\,dx\right|&\leq\int_{\mathbb{R}^{n}}|w'_{\lambda_i}(|y|)w'_{\lambda_l}(|y+q_i-q_l|)|\leq C \int_{\mathbb{R}^{n}}\frac{dy}{(1+|y|)^{\nu_1}(1+|y+q_i-q_\ell|)^{\nu_1}}.
\end{align*}
Splitting this integral we can check that, since $\nu_1>n$,
\begin{align*}
\int_{\mathbb{R}^{n}\backslash B_{\frac{\eta}{2}}(0)}\frac{dy}{(1+|y|)^{\nu_1}(1+|y+q_i-q_\ell|)^{\nu_1}}&\leq C\eta^{-\nu_1}\int_{\mathbb{R}^{n}}\frac{dy}{(1+|y+q_i-q_\ell|)^{\nu_1}}\leq C\eta^{-\nu_1},
\end{align*}
and 
\begin{align*}
\int_{B_{\frac{\eta}{2}}(0)}\frac{dy}{(1+|y|)^{\nu_1}(1+|y+q_i-q_\ell|)^{\nu_1}}&\leq C\eta^{-\nu_1}\int_{B_{\frac{\eta}{2}}(0)}\frac{dy}{(1+|y|)^{-\nu_1}}\leq C\eta^{-\nu_1},
\end{align*}
where we have used that $|y+q_i-q_\ell|\geq |q_i-q_\ell|-|y|\geq \eta-\frac{\eta}{2}=\frac{\eta}{2}$ for $y\in B_{\frac{\eta}{2}}(0)$.Thus,
\begin{equation}\label{eq:iljm3}
\int_{\mathbb{R}^{n}}Z_{ij}Z_{\ell m}\,dx=O(\eta^{-\nu_1}).
\end{equation}
The result follows putting together \eqref{eq:iljm1}, \eqref{eq:iljm2} and \eqref{eq:iljm3}.
%\textbf{Case 4:} $i\neq l$ and $j\neq m.$
%\begin{align*}
%\int_{\mathbb{R}^{n}}Z_{ij}Z_{lm}&=\int_{\mathbb{R}^{n}}w'_{\lambda_i}(|y|)w'_{\lambda_l}(|y+q_i-q_l|)\frac{y_j}{|y|}\frac{y_m+q_{im}-q_{lm}}{|y+q_i-q_l|}\\
%&=\int_{\mathbb{R}^{n-1}}(y_m+q_{im}-q_{lm})\int_{\mathbb{R}}w'_{\lambda_i}(|y|)w'_{\lambda_l}(|y+q_i-q_l|)\frac{y_j}{|y||y+q_i-q_l|}dy_jdy'\\
%&=0,
%\end{align*}
%since the function $w'_{\lambda_i}(|y|)w'_{\lambda_l}(|y+q_i-q_l|)\frac{y_j}{|y||y+q_i-q_l|}$ is odd.
\end{proof}

\begin{corollary} \label{orthogonal1}
The functions $Z_{ij}$ satisfy, for every $i,\ell=1,\ldots, k$, and every $j,m=1,\ldots,n$, 
\[\int_{\Omega_\varepsilon}Z_{ij}Z_{\ell m}\,dx=\alpha_i\delta_{i\ell}\delta_{jm}+O(\eta^{-\nu_1}+\varepsilon^{2\nu_1-n}),\]
where $\nu_1$ and $\alpha_i$ are given in Lemma \ref{Zij} and Lemma \ref{orthogonal} respectively.
\begin{proof}
By Lemma \ref{orthogonal} we have that
\begin{align*}
  \int_{\Omega_\varepsilon}Z_{ij}Z_{\ell m}\,dx&= \int_{\mathbb{R}^{n}}Z_{ij}Z_{\ell m}\,dx-\int_{\mathbb{R}^{n}\backslash\Omega_\varepsilon}Z_{ij}Z_{\ell m}\,dx =\alpha_i\delta_{ijm\ell }+O(\eta^{-\nu_1})-\int_{\mathbb{R}^{n}\backslash\Omega_\varepsilon}Z_{ij}Z_{\ell m}\, dx,
\end{align*}
and, using Lemma \ref{Zij}, 
\[\left|\int_{\mathbb{R}^{n}\backslash\Omega_\varepsilon}Z_{ij}Z_{\ell m}\,dx\right|\leq C\int_{\mathbb{R}^{n}\backslash\Omega_\varepsilon}|x-q_i|^{-\nu_1}|x-q_\ell|^{-\nu_1}\,dx \leq C\varepsilon^{2\nu_1-n},\]
which gives the desired result.
\end{proof}
\end{corollary}

\begin{corollary}\label{cor:derw_qij}
Let $q_{ij}$ denote the $j$-th coordinate of $q_i$. Then
$$\frac{\partial w_i}{\partial q_{ij}}(x)=-\frac{\partial w_i}{\partial x_j}(x)+O\bigg(\frac{\varepsilon}{(1+|x-q_i|)^{\nu_1-1}}\bigg)\qquad\mbox{ for every }i=1,\ldots, k,\quad j=1,\ldots,n,$$
with $\nu_1$ given in Lemma \ref{Zij}.
\end{corollary}
\begin{proof}
By a straightforward computation we get
$$\frac{\partial w_i}{\partial q_{ij}}=-\frac{\partial w_i}{\partial x_j}+\frac{1}{p-1}\lambda_i^{\frac{1}{p-1}-1}\frac{\partial \lambda_i}{\partial q_{ij}}w(\lambda_i^{\frac{1}{2s}}(x-q_i))+\frac{1}{2s}\lambda_i^{\frac{1}{p-1}+\frac{1}{2s}-1}\frac{\partial \lambda_i}{\partial q_{ij}}\big[\nabla w(\lambda_i^{\frac{1}{2s}}(x-q_i))\cdot (x-q_i)\big].$$
Noticing that
$$\frac{\partial \lambda_i}{\partial q_{ij}}=\varepsilon \frac{\partial \lambda_i}{\partial \xi_{ij}}=\varepsilon \frac{\partial V(\xi_i)}{\partial \xi_{ij}},$$
and using Remark \ref{remark2}, \cite[Claim (5.5)]{DdPDV2015} and Remark \ref{unifV} the result follows.
\end{proof}

\section{Linear theory}
\label{Section 5}
Let $\delta_*\in(0,1)$ and $\eta\gg 1$ fixed and consider the configuration space
\begin{equation}\label{Xi}
  \Xi_\eta:=\Big\{q=(q_1,\ldots,q_k):\, \min\{\mbox{dist}(q_i,\partial\Omega_\varepsilon)\}\geq\frac{\delta_*}{\varepsilon},\;\eta:=\min\limits_{i\neq \ell}|q_i-q_\ell|\Big\}.
\end{equation}
For every $q\in \Xi_\eta$ and every
\begin{equation}\label{mu}
\frac{n}{2}<\mu< \frac{n+2s}{2},
\end{equation}
we define
\begin{equation}\label{rho}
\|\phi\|_*:=\|\rho_q^{-1}\phi\|_{L^\infty(\R^n)}\quad \mbox{ where }\quad \rho_q(x):=\sum\limits_{i=1}^k\frac{1}{(1+|x-q_i|)^{\mu}}.
\end{equation}
\begin{remark}\label{unifV}
Since $V\in C^2(\overline{\Omega})$, $\nabla V$ and $D^2V$ are uniformly bounded in $\Omega$.
 \end{remark}
Let us start considering the general problem
\begin{equation}\label{generalp}
\left\{\begin{array} {ll}
(-\Delta)^s\phi+\mathcal{W}\phi+g=0 \,\, & \mbox{in} \, \Omega_\varepsilon,\\
  \phi=0  \,\, &\mbox{in} \, \mathbb{R}^n\setminus\Omega_\varepsilon,
  \end{array}\right.
\end{equation}
where $g\in L^2(\R^n)\cap L^\infty(\R^n)$ and $\mathcal{W}$ is a bounded positive potential.
Replicating \cite[Lemma 6.1]{DdPDV2015} with the obvious adaptations due to $\mathcal{W}$ one can obtain the following regularity results:

\begin{lemma}\label{holderest} Let $\phi\in H^s(\mathbb{R}^n)$ be a solution to \eqref{generalp} and assume  
\begin{equation}\label{eq:general_W}
\kappa:=\inf\limits_{x\in\Omega_\varepsilon}\mathcal{W}>0.
\end{equation}
Then there exists a positive constant $C$, depending only on $\kappa$ and $\|\mathcal{W}\|_{L^\infty(\Omega_\varepsilon)}$, such that
\[\|\phi\|_{L^\infty(\mathbb{R}^n)}+\sup\limits_{x\neq y}\frac{|\phi(x)-\phi(y)|}{|x-y|^s}\leq C(\|g\|_{L^2(\mathbb{R}^n)}+\|g\|_{L^\infty(\mathbb{R}^n)}).\]
\end{lemma}

\begin{lemma}\label{generest0}
Suppose $\|g\|_{*}<\infty,$ and let $\phi\in H^s(\mathbb{R}^n)$ be a solution to \eqref{generalp}. Let $r>0$ and $B:=\cup_{i=1}^k B_r(q_i)\subset \Omega_\varepsilon$ such that
$$\inf_{x\in\Omega_\varepsilon\setminus B}\mathcal{W}(x)>0.$$ 
Then
 \[\|\phi\|_*\leq C(\|\phi\|_{L^\infty(B)}+\|g\|_{*}),\]
 for a positive constant $C$ depending on $n, s, \Omega$ and $\mathcal{W}.$
\end{lemma}
We closely follow the approach of \cite[Lemma 6.2]{DdPDV2015}, so we only sketch the proof.
\begin{proof} Denote $\omega:=\inf_{x\in\Omega_\varepsilon\setminus B}\mathcal{W}(x)>0$. It can be checked that
\[(-\Delta)^s\phi+\widetilde{\mathcal{W}}\phi=\tilde{g}\]
where
\[\widetilde{\mathcal{W}}:=\omega\chi_B+\mathcal{W}(1-\chi_B), \qquad \tilde{g}:=(\omega-\mathcal{W})\chi_B\phi-g.\]
Hence $\widetilde{\mathcal{W}}\geq\omega$
and
\begin{equation}\label{ineqtg}
\|\tilde{g}\|_{*}\leq\sup\limits_{x\in B}|\rho_q^{-1}(\omega+\mathcal{W})\phi|
+\|g\|_{*}\leq C(1+r)^\mu\|\mathcal{W}\|_{L^\infty(\Omega_\varepsilon)}\|\phi\|_{L^\infty(B)}+\|g\|_{*}\leq C\|\phi\|_{L^\infty(B)}+\|g\|_{*},
\end{equation}
with $C$ depending on $r,\mu$ and $\|\mathcal{W}\|_{L^\infty(\Omega_\varepsilon)}.$

Consider the solution $\tilde{\phi}\in H^s(\mathbb{R}^n)$ (see \cite[(2.4)]{DdPW2014}) to the problem
\[(-\Delta)^s\tilde{\phi}+\omega\tilde{\phi}=(1+|x|)^{-\mu},\]
which is nonnegative and satisfies
\begin{equation}\label{suptphi}
  \sup\limits_{x\in\mathbb{R}^n}(1+|x|)^{\mu}\tilde{\phi}\leq C\sup\limits_{x\in\mathbb{R}^n}(1+|x|)^{\mu}(1+|x|)^{-\mu}=C,
\end{equation}
for some positive constant $C$ depending on $\omega$ (see \cite[Lemma 2.2, Lemma 2.4]{DdPW2014}).
Thus
\[(\widetilde{\mathcal{W}}-\omega)\tilde{\phi}(x-q_i)\geq0, \qquad \forall \, i=1,\ldots,k.\]
Denote 
$$\tilde{\phi}_q(x):=\|\tilde{g}\|_{*}\sum\limits_{i=1}^k\tilde{\phi}(x-q_i)\geq0,$$ 
and $\psi:=\tilde{\phi}_q\pm\phi.$ Then it is easy to check that 
\begin{align*}
 (-\Delta)^s\psi+\widetilde{\mathcal{W}}\psi&\geq 0\qquad \mbox{ in }\Omega_\varepsilon.
\end{align*}
Since $\psi=\tilde{\phi}_q\geq 0$ in $\mathbb{R}^n\setminus\Omega_\varepsilon$, by the maximum principle we conclude $\psi\geq 0$ in $\mathbb{R}^n$. Recalling \eqref{ineqtg} and \eqref{suptphi}, for any $x\in \mathbb{R}^n$ we have
\begin{align*}
 \mp \rho_q^{-1}\phi&= \rho_q^{-1}(\tilde{\phi}_q-\psi) \leq \rho_q^{-1}\tilde{\phi}_q\leq \|\tilde{g}\|_{*}\sup\limits_{y\in\mathbb{R}^n}(1+|y|)^{\mu}\tilde{\phi}(y)\leq  C (\|\phi\|_{L^\infty(B)}+\|g\|_{*}),
\end{align*}
and the proof is concluded.
\end{proof}
Fixing $B=\emptyset$ in Lemma \ref{generest0} we immediately get the following:
\begin{corollary}\label{generest}
Suppose $\|g\|_{*}<\infty,$ and let $\phi\in H^s(\mathbb{R}^n)$ be a solution to \eqref{generalp}. Assume also
\[\inf\limits_{x\in\Omega_\varepsilon}\mathcal{W}(x)>0.\]
Then there exists a constant $C>0$, depending on $n$, $s$, $\Omega$ and $\|\mathcal{W}\|_{L^\infty(\Omega_\varepsilon)}$, such that
\[\|\phi\|_{*}\leq C\|g\|_{*}.\]
\end{corollary}

Define the Hilbert space
 \[X:=\bigg\{\phi\in H^s(\mathbb{R}^n):\phi=0 \,\, \mbox{in}\,  \mathbb{R}^n\setminus\Omega_\varepsilon,\;\; \int_{\Omega_\varepsilon}\phi Z_{ij}=0\,\;\; \mbox{for all}\;\,\, i=1,\ldots,k,\;\,j=1,\ldots, n \bigg\},\]
 with $Z_{ij}$ defined in \eqref{deZij}.
Then, as explained in the introduction, we will look for a solution to \eqref{dirichlet1} in the form
\[u=U_q+\phi,\]
where $U_q$ is given in \eqref{Wq}, $q\in \Xi_\eta$, and $\phi\in X$ is a small function as long as $\varepsilon$ small enough. In terms of $\phi$, $\eqref{dirichlet1}$ becomes
\begin{equation}\label{dirichlet3}
(-\Delta)^{s} \phi + V(\varepsilon x)\phi -p W_q^{p-1}\phi=E(\phi)+N(\phi)\qquad \mbox{in} \; \Omega_\varepsilon,
\end{equation}
where $W_q$ was defined in \eqref{Wq} and
\begin{align}
  E(\phi)&:=(U_q+\phi)^{p}-(W_q+\phi)^{p}+\sum\limits_{i=1}^k(\lambda_i-V(\varepsilon x))\bar{u}_i+\bigg(\sum\limits_{i=1}^kw_i\bigg)^{p}-\sum\limits_{i=1}^kw_i^{p},\label{E}\\
  N(\phi)&:=(W_q+\phi)^{p}-p W_q^{p-1}\phi-W_q^{p}\label{N}.
\end{align}
Notice that $E(\phi)$ reflects the three effects described in the introduction. The first subtraction comes from the boundary correction, the second from the presence of a non constant potential, and the third from the interaction among peaks. 

To avoid the kernel of the linearized operator, instead of solving \eqref{dirichlet3} we will consider first its projected version:
\begin{equation}\label{dirichletProj}
(-\Delta)^{s} \phi + V(\varepsilon x)\phi -p W_q^{p-1}\phi=E(\phi)+N(\phi)+\sum\limits_{i=1}^k\sum\limits_{j=1}^n c_{ij}Z_{ij}\qquad \mbox{in} \; \Omega_\varepsilon,
\end{equation}
for some coefficients $c_{ij}$. We will develop a solvability theory for the associated linear problem, and then in Section \ref{sec:nonlinear} we will handle \eqref{dirichletProj} by means of a fixed point argument. Hence, let us consider
\begin{equation}\label{linearp}
\left\{\begin{array} {ll}
 (-\Delta)^s\phi+V(\varepsilon x)\phi-pW_q^{p-1}\phi+g(x)=\sum\limits_{i=1}^k\sum\limits_{j=1}^n c_{ij}Z_{ij} & \mbox{in }\; \Omega_\varepsilon,\\
 \phi=0 & \mbox{in }\; \mathbb{R}^{n}\setminus\Omega_\varepsilon,\\
 \displaystyle \int_{\Omega_\varepsilon} \phi Z_{ij}=0 & i\in\{1,\ldots,k\},\;\,j\in\{1,\ldots, n\},
\end{array}\right.
\end{equation}
where $g\in L^2(\R^n)\cap L^\infty(\R^n)$.
\begin{lemma}\label{cij}
Let $q\in \Xi_\eta$. The coefficients $c_{ij}$ appearing in \eqref{linearp} satisfy
\[c_{ij}=\frac{1}{\alpha_i}\int_{\mathbb{R}^{n}}Z_{ij}g\,dx+\tilde{f}_{ij},\qquad \alpha_i=\int_{\R^n} Z_{i1}^2\,dx,\]
with
\[|\tilde{f}_{ij}|\leq C(\varepsilon^{\frac{1}{2}}+\eta^{-\min\{1,p-1\}(n+2s)})(\|\phi\|_{L^2(\mathbb{R}^n)}+\|g\|_{L^2(\mathbb{R}^n)}),\]
and $C$ a positive constant independent of $q$.%\tau=\min\{\frac{1}{2},\frac{\nu_1}{2}\}=\frac{1}{2},
\end{lemma}
\begin{proof}
Fix $\ell\in\{1,\ldots,k\}$ and $c>0$ such that $B_{c/\varepsilon}(q_\ell)\subseteq \Omega_\varepsilon$. Consider $\tau_{\varepsilon,\ell}\in C^\infty(\R^n,[0,1])$ such that 
$$\tau_{\varepsilon,\ell}=1\mbox{ in }B_{c/\varepsilon-1}(q_\ell),\quad \tau_{\varepsilon,\ell}=0\mbox{ in }\R^n\setminus B_{c/\varepsilon-1}(q_\ell),\quad|\nabla \tau_{\varepsilon,\ell}|\leq C,$$ 
for some positive constant $C$. Let $T_{\ell m}:= \tau_{\varepsilon,\ell} Z_{\ell m}$. Multiplying \eqref{linearp} by $T_{\ell m}$, we get
\[\sum\limits_{i=1}^k\sum\limits_{j=1}^n c_{ij}\int_{\Omega_\varepsilon}Z_{ij}T_{\ell m}\,dx=\int_{\Omega_\varepsilon}[(-\Delta)^s\phi+V(\varepsilon x)\phi-pW_q^{p-1}\phi+g]T_{\ell m}\,dx.\]
Proceeding like in \cite[Proof of Lemma 7.2]{DdPDV2015} it can be checked that 
$$\|T_{\ell m}-Z_{\ell m}\|_{H^2(\R^n)}\leq C\varepsilon^{\frac n 2},$$
and hence
$$\int_{\Omega_\varepsilon}(-\Delta)^s\phi\, T_{\ell m}\,dx=\int_{\Omega_\varepsilon}\phi\, (-\Delta)^s Z_{\ell m}\,dx+O\big(\varepsilon^{\frac{n}{2}}\|\phi\|_{L^2(\R^n)}\big).$$
Using that $Z_{\ell m}$ satisfies
\begin{equation}\label{eqZij}
(-\Delta)^sZ_{\ell m}+\lambda_\ell Z_{\ell m}=pw_\ell^{p-1}Z_{\ell m},
\end{equation}
one gets
\begin{equation*}\begin{split}
\int_{\Omega_\varepsilon}[(-\Delta)^s\phi+V(\varepsilon x)\phi&-pW_q^{p-1}\phi]T_{\ell m}\,dx\\
&= \int_{\Omega_\varepsilon}[(-\Delta)^sZ_{\ell m}+V(\varepsilon x)Z_{\ell m}-pW_q^{p-1}Z_{\ell m}]\phi\,dx+O(\varepsilon^{\frac{n}{2}}\|\phi\|_{L^2(\R^n)})\\
&=\int_{\Omega_\varepsilon}[pw_\ell^{p-1}-pW_q^{p-1}+V(\varepsilon x)-\lambda_\ell]Z_{\ell m}\phi\,dx+O(\varepsilon^{\frac{n}{2}}\|\phi\|_{L^2(\R^n)}).
\end{split}\end{equation*}
Likewise,
$$\int_{\Omega_\varepsilon}gT_{\ell m}\,dx=\int_{\R^n}g Z_{\ell m}\,dx -\int_{\R^n\setminus\Omega_\varepsilon}g Z_{\ell m}\,dx+O\big(\varepsilon^{\frac{n}{2}}\|g\|_{L^2(\R^n)}\big),$$
and therefore
\begin{align*}
\sum\limits_{i=1}^k\sum\limits_{j=1}^n c_{ij}\int_{\Omega_\varepsilon}Z_{ij}T_{\ell m}\,dx =\int_{\mathbb{R}^n}gZ_{\ell m}\, dx+f_{\ell m}+O(\varepsilon^{\frac{n}{2}}(\|\phi\|_{L^2(\R^n)}+\|g\|_{L^2(\R^n)})),
\end{align*}
where
\[f_{\ell m}:=\int_{\Omega_\varepsilon}(pw_\ell^{p-1}-pW_q^{p-1}+V(\varepsilon x)-\lambda_\ell)Z_{\ell m}\phi\,dx-\int_{\mathbb{R}^n\setminus\Omega_\varepsilon}gZ_{\ell m}\,dx.\]
Let us estimate $f_{\ell m}.$ By H\"{o}lder's inequality and Lemma \ref{Zij}, 
\[\bigg|\int_{\mathbb{R}^n\setminus\Omega_\varepsilon}gZ_{\ell m}\,dx\bigg|\leq\|g\|_{L^2(\mathbb{R}^n)}
\left(\int_{\mathbb{R}^n\setminus\Omega_\varepsilon}\frac{C}{|x-q_\ell|^{2\nu_1}}\,dx\right)^{\frac{1}{2}}\leq C\varepsilon^{\nu_1-\frac{n}{2}}\|g\|_{L^2(\mathbb{R}^n)},\]
for a positive $C$ independent of $q$, and $\nu_1:=\min\{(n+2s+1),p(n+2s)\}$. Using \cite[Lemma 7.11]{DdPDV2015}, Lemma \ref{Zij} and \eqref{behavwi} and the fact that
$$|V(\varepsilon x)-\lambda_\ell|\leq C\varepsilon|x-q_\ell|,$$
with $C>0$ independent of $q\in \Xi_\eta$, we can bound
\begin{align*}
|(pw_\ell^{p-1}-pW_q^{p-1}+V(\varepsilon x)-\lambda_\ell)Z_{\ell m}|&\leq (|pw_\ell^{p-1}-pW_q^{p-1}|+|V(\varepsilon x)-\lambda_\ell|)|Z_{\ell m}|\\
&\leq C\Big(\sum\limits_{h\neq \ell}w_h^{r}+\varepsilon|x-q_\ell|\Big)(1+|x-q_\ell|)^{-\nu_1}\\
  &\leq C \Big(\sum\limits_{h \neq \ell}(1+|x-q_h|)^{-r(n+2s)}+\varepsilon^{\frac{1}{2}}|x-q_\ell|^{1/2})(1+|x-q_\ell|)^{-\nu_1},
\end{align*}
where $r:=\min\{1,p-1\}$. Then
\begin{align*}
\int_{\Omega_\varepsilon}|&(pw_\ell^{p-1}-pW_q^{p-1}+V(\varepsilon x)-\lambda_\ell)Z_{\ell m}\phi|\,dx \\
  & \leq C\bigg(\int_{\Omega_\varepsilon}\bigg(\sum\limits_{h\neq \ell}(1+|x-q_h|)^{-r(n+2s)}\bigg)^2(1+|x-q_\ell|)^{-2\nu_1}\,dx\bigg)^\frac{1}{2}\bigg(\int_{\Omega_\varepsilon}|\phi|^2\,dx\bigg)^\frac{1}{2}\\
  &\quad+C\bigg(\int_{\Omega_\varepsilon}\varepsilon|x-q_\ell|(1+|x-q_\ell|)^{-2\nu_1}\,dx\bigg)^\frac{1}{2}\bigg(\int_{\Omega_\varepsilon}|\phi|^2\,dx\bigg)^\frac{1}{2}\\
  &\leq C (\eta^{-r(n+2s)}+\varepsilon^{\frac{1}{2}})\|\phi\|_{L^2(\mathbb{R}^n)},
\end{align*}
%\begin{align*}
%|(p&w_i^{p-1}-pW_q^{p-1}+V(\varepsilon x)-\lambda_i)Z_{ij}|\leq (|pw_i^{p-1}-pW_q^{p-1}|+|V(\varepsilon x)-\lambda_i|)|Z_{ij}|\\
%&\leq C\Big(\sum\limits_{\ell\neq i}w_\ell^{r}+\varepsilon|x-q_i||\nabla V(\xi_i)|+\varepsilon^\alpha \|D^2V\|_{L^\infty(\Omega)}|x-q_i|^\alpha\Big)(1+|x-q_i|)^{-\nu_1}\\
%  &\leq C \Big(\sum\limits_{\ell \neq i}(1+|x-q_\ell|)^{-r(n+2s)}+\varepsilon|x-q_i||\nabla V(\xi_i)|+\varepsilon^\alpha \|D^2V\|_{L^\infty(\Omega)}|x-q_i|^\alpha)(1+|x-q_i|)^{-\nu_1},
%\end{align*}
%where $\alpha=\min\{2,n+2s\}$, $r:=\min\{1,p-1\},$ and $C>0$ only depends on $V.$ Then
%\begin{align*}
%\int_{\Omega_\varepsilon}|&(pw_i^{p-1}-pW_q^{p-1}+V(\varepsilon x)-\lambda_i)Z_{ij}\phi|\,dx \\
%  & \leq C\bigg(\int_{\Omega_\varepsilon}\bigg(\sum\limits_{\ell\neq i}(1+|x-q_\ell|)^{-r(n+2s)}\bigg)^2(1+|x-q_i|)^{-2\nu_1}\,dx\bigg)^\frac{1}{2}\bigg(\int_{\Omega_\varepsilon}|\phi|^2\,dx\bigg)^\frac{1}{2}\\
%  &\quad+\left(\int_{\Omega_\varepsilon}(\varepsilon|x-q_i||\nabla V(\xi_i)|)^2(1+|x-q_i|)^{-2\nu_1}\,dx\right)^\frac{1}{2}\left(\int_{\Omega_\varepsilon}|\phi|^2\,dx\right)^\frac{1}{2}\\
%  \\
%  &\quad+\left(\int_{\Omega_\varepsilon}(\varepsilon^\alpha \|D^2V\|_{L^\infty(\Omega)}|x-q_i|^\alpha)^2(1+|x-q_i|)^{-2\nu_1}\,dx\right)^\frac{1}{2}\left(\int_{\Omega_\varepsilon}|\phi|^2\,dx\right)^\frac{1}{2}\\
%  &\leq C (\eta^{-r(n+2s)}+\varepsilon^{1/2}|\nabla V(\xi_i)|)\|\phi\|_{L^2(\mathbb{R}^n)},
%\end{align*}
since
\[\int_{\Omega_\varepsilon}\Bigg(\sum\limits_{h\neq \ell}(1+|x-q_h|)^{-r(n+2s)}\Bigg)^2(1+|x-q_\ell|)^{-2\nu_1}\,dx\leq C \eta^{-2r(n+2s)}\]
and $2\nu_1-1>n$, so that
\[\int_{\R^n}|x-q_\ell|(1+|x-q_\ell|)^{-2\nu_1}\,dx\leq C.\]
Thus
\[
|f_{\ell m}|\leq C (\eta^{-r(n+2s)}+\varepsilon^{\frac{1}{2}}+\varepsilon^{\nu_1-\frac{n}{2}})(\|\phi\|_{L^2(\mathbb{R}^n)}+\|g\|_{L^2(\mathbb{R}^n)})\leq C(\eta^{-r(n+2s)}+\varepsilon^{\frac{1}{2}})(\|\phi\|_{L^2(\mathbb{R}^n)}+\|g\|_{L^2(\mathbb{R}^n)}).
\]
On the other hand, from Corollary \ref{orthogonal1},
\[
\int_{\Omega_\varepsilon}Z_{ij}T_{\ell m}\,dx =\int_{\Omega_\varepsilon}Z_{ij}Z_{\ell m}\,dx +O(\varepsilon^{\frac{n}{2}})=\alpha_i\delta_{i\ell}\delta_{jm}+O(\eta^{-r(n+2s)}+\varepsilon^{\frac{1}{2}}).
\]
Mimicking the argument in \cite[Proof of Lemma 7.2]{DdPDV2015} we can conclude that
\[
c_{ij}=\frac{1}{\alpha_i}\int_{\R^n}gZ_{ij}\,dx+\frac{1}{\alpha_i}f_{ij}+O((\eta^{-r(n+2s)}+\varepsilon^{\frac{1}{2}})(\|\phi\|_{L^2(\R^n)}+\|g\|_{L^2(\R^n)})),
\]
and the result follows.
\end{proof}

The rest of this section will be devoted to establishing the next existence result for the linear problem \eqref{linearp}:
 \begin{proposition}\label{exist} Let $g\in L^2(\mathbb{R}^n)$ with $\|g\|_{*}<\infty$ and $q\in\Xi_\eta$. If $\varepsilon$ is small enough, then there exists a unique solution $\phi\in X$ to the problem \eqref{linearp}. Moreover, there exists a constant $C>0$, independent of $q$, such that
 \[\|\phi\|_{*}\leq C \|g\|_{*}.\]
 \end{proposition}
We will prove this result in two steps: first the {\it a priori} estimate, and then the existence of solution. 
 \begin{lemma}\label{priestimate} Let $g\in L^2(\mathbb{R}^n)$ with $\|g\|_{*}<\infty$ and $q\in\Xi_\eta$. If $\varepsilon$ is small enough and $\phi$ is a solution  of \eqref{linearp}, then there exists a constant $C>0$, independent of $q$, such that
 \[\|\phi\|_{*}\leq C \|g\|_{*}.\]
 \end{lemma}
 \begin{proof}
Assume by contradiction that there exist sequences $\varepsilon_m$ converging to $0$ as $m\to\infty$, $q_i^m:=\frac{\xi_i}{\varepsilon_m}$, $i=1,\ldots,k$ with
 $\min\{|q_i^m-q_\ell^m|:i\neq \ell\}\rightarrow\infty,$ and $\phi_m,g_m$ satisfying \eqref{linearp} with
 \begin{equation}\label{assumption}
\|\phi_m\|_{*,m}=1, \qquad \|g_m\|_{*,m}\rightarrow 0\quad \mbox{ as } m\to\infty,
\end{equation}
 where 
 $$\rho_{q_m}(x):=\sum\limits_{i=1}^{k}\frac{1}{(1+|x-q_i^m|^\mu)},\qquad \frac{n}{2}<\mu<\frac{n+2s}{2}, \qquad \|\phi\|_{*,m}:=\|\rho_{q_m}^{-1}\phi\|_{L^\infty(\R^n)}.$$
 Using \eqref{assumption}, Lemma \ref{cij}, Lemma \ref{Zij} and \eqref{behavwi}, one can prove that there exists $C>0$ such that
$$\bigg\|-pW_q^{p-1}\phi_m+g_m-\sum\limits_{i=1}^k\sum\limits_{j=1}^n c^m_{ij}Z^m_{ij}\bigg\|_{L^\infty(\R^n)}\leq C,\quad \quad \bigg\|-pW_q^{p-1}\phi_m+g_m-\sum\limits_{i=1}^k\sum\limits_{j=1}^n c^m_{ij}Z^m_{ij}\bigg\|_{L^2(\R^n)}\leq C,$$
and hence, by Lemma \ref{holderest}, the $\phi_m$ are equicontinuous.
 
For any fixed $R>0,$ we claim that
 \begin{equation}\label{claim}
   \sum\limits_{i=1}^{k}\|\phi_m\|_{L^\infty(B_R(q_i^m))}\rightarrow 0\,\, \quad \mbox{as} \quad \, m\rightarrow\infty.
\end{equation}
Indeed, suppose that there exist $\gamma>0$ and $m_0\in\mathbb{N}$ such that, for some fixed $i$, $\|\phi_m\|_{L^\infty(B_R(q_i^m))}\geq\gamma$ for every $m\geq m_0.$ For such $i$, define
 \[\tilde{\phi}_m(x):=\phi_m(x+q_i^m),\qquad \tilde{\Omega}_m:=\{x=y-q_i^m:\, y\in\Omega_{\varepsilon_m}\},\]
and assume $\lambda_i^m=V(\varepsilon_m q_i^m)\rightarrow\tilde{\lambda}>0$. Hence $\tilde{\phi}_m$ satisfies
\begin{equation}\label{sequencep}
\left\{\begin{array} {ll}
 (-\Delta)^s\tilde{\phi}_m+V(\varepsilon_m(q_i^m+x))\tilde{\phi}_m-p\big(w_{\lambda_i^m}+\theta_m(x)\big)^{p-1}\tilde{\phi}_m+\tilde{g}_m=0 & \mbox{in }\; \tilde{\Omega}_m,\\
 \tilde{\phi}_m=0 & \mbox{in }\; \mathbb{R}^{n}\setminus\tilde{\Omega}_m,\\
 \int_{\tilde{\Omega}_m}\tilde{\phi}_m \tilde{Z}_{\ell j}^m dx=0 & \ell\in\{1,\ldots,k\},\;\,j\in\{1,\ldots, n\},
\end{array}\right.
\end{equation}
 where 
 $$\tilde{Z}_{\ell j}^m(x) :=\frac{\partial w_{\lambda_\ell^m}(x+q_i^m-q_\ell^m)}{\partial x_j},\qquad \theta_m(x):=\sum\limits_{\ell\neq i} w_{\lambda_\ell^m}(x+q_i^m-q_\ell^m),$$
  and $$\tilde{g}_m(x):=g_m(x+q_i^m)-\sum\limits_{j=1}^{n}\sum\limits_{\ell=1}^{k}c_{\ell j}^m\partial_{x_j} w_{\lambda_\ell^m}(x+q_i^m-q_\ell^m).$$
Notice that, using the decay properties of $w$, \eqref{assumption} and the definition of $q_i^m$, 
$$\theta_m\to 0\quad \mbox{ and }\quad \tilde{g}_m\to 0\quad \mbox{ uniformly on compact sets}.$$
Since $q_i^m\in \Xi_\eta$, we have that $B_{\delta_*/\varepsilon_m}(q_i^m)\subset\Omega_{\varepsilon_m}.$ Then $B_{\delta_*/\varepsilon_m}(0)\subset\tilde{\Omega}_m$, which implies that $\tilde{\Omega}_m$ converges to $\mathbb{R}^n$ as $m\rightarrow\infty.$

 Moreover we have that
\[
   \|\tilde{\phi}_m\|_{L^\infty(B_R(0))}\geq\gamma \quad \mbox{and}\quad \|\rho_m^{-1}(\cdot+q_i^m)\tilde{\phi}_m\|_{L^\infty(\mathbb{R}^n)}=1 .\]
Since $\{\phi_m\}$ is equicontinuous,  so it is $\{\tilde{\phi}_m\}$. Then, up to a subsequence, $\tilde{\phi}_m\rightarrow\tilde{\phi}$ uniformly on a compact set. Here, $\tilde{\phi}\in L^2(\mathbb{R}^n)$ by Fatou's theorem, \eqref{assumption} and the fact that $\mu>\frac{n}{2}.$

In addition,
\begin{gather}\label{limitest} \|\tilde{\phi}\|_{L^\infty(B_R(0))}\geq\gamma \quad \mbox{and}\quad \|\rho_m^{-1}(\cdot+q_i^m)\tilde{\phi}\|_{L^\infty(\mathbb{R}^n)}\leq1,
\end{gather}
and it can be seen that $\tilde{\phi}$ solves the equation
\[(-\Delta)^s\tilde{\phi}+\tilde{\lambda}\tilde{\phi}-pw_{\tilde{\lambda}}^{p-1}\tilde{\phi}=0\quad \mbox{in}\,\, \mathbb{R}^n,\]
in weak sense (see \cite[Proof of Lemma 7.3]{DdPDV2015}), and then in  a strong sense (see \cite{SV2014}). Thus, by \cite[Theorem 3]{FLS2016}, 
$$\tilde{\phi}\in\mbox{span}\bigg\{\frac{\partial w_{\tilde{\lambda}}}{\partial x_1},\ldots, \frac{\partial w_{\tilde{\lambda}}}{\partial x_n}\bigg\},$$
and passing to the limit in the orthogonality condition \eqref{sequencep} we conclude that $\tilde{\phi}\equiv0,$ which contradicts \eqref{limitest}. Therefore \eqref{claim} holds.

Taking $R$ large enough, by Lemma \ref{generest0} one has
\begin{align*}
  \|\phi_m\|_{*,m} &\leq C\bigg(\|\phi_m\|_{L^\infty(B_R(q_i^m))}+\Big\|g_m+\sum\limits_{ij}c_{ij}^mZ_{ij}^m\Big\|_{*,m} \Bigl) \leq C\bigg(\|\phi_m\|_{L^\infty(B_R(q_i^m))}+\|g_m\|_{*,m}+\sum\limits_{ij}|c_{ij}^m| \bigg),
\end{align*}
where in the second inequality we applied Lemma \ref{Zij} and the fact that $\mu<n+2s.$ Then using \eqref{assumption}, \eqref{claim} and Lemma \ref{cij}, 
\[\|\phi_m\|_{*,m}\rightarrow 0 \quad \mbox{as}\,\, m\rightarrow\infty,\]
a contradiction with \eqref{assumption}.
 \end{proof}

Consider the auxiliary problem
 \begin{equation}\label{auxiliaryp}
\left\{\begin{array} {ll}
 (-\Delta)^s\phi+V(\varepsilon x)\phi+h(x)=\sum\limits_{i=1}^k\sum\limits_{j=1}^n c_{ij}Z_{ij} & \mbox{in }\; \Omega_\varepsilon,\\
 \phi=0 & \mbox{in }\; \mathbb{R}^{n}\setminus\Omega_\varepsilon,\\
 \int_{\Omega_\varepsilon} \phi Z_{ij}=0 & i=1,\ldots,k,\;\,j=1,\ldots, n,
\end{array}\right.
\end{equation}
where $h\in L^2(\R^n)\cap L^\infty(\R^n)$. Then we can prove the following existence result:

 \begin{lemma}\label{existauxil} Let $h\in L^2(\mathbb{R}^n)$ with $\|h\|_{*}<\infty$ and $q\in \Xi_\eta$. Then the problem \eqref{auxiliaryp} has a unique solution $\phi\in X$, that satisfies
 \[\|\phi\|_{*}\leq C \|h\|_{*},\]
 for a positive constant $C$ independent of $q$.
 \end{lemma}
 
 The proof of this result follows, under straightforward adaptations, like in \cite[Proposition 7.4]{DdPDV2015}, as a consequence of Riesz theorem (thanks to hypohtesis \eqref{eq:infBoundV}), Corollary \ref{generest} and Lemma \ref{cij}. We omit the proof.

Define the Banach space
\begin{equation}\label{defY}
Y:=\{\phi:\mathbb{R}^{n}\rightarrow\mathbb{R}\;: \;\|\phi\|_*<+\infty\}.
\end{equation}
We will solve problem \eqref{linearp} by using Lemma \ref{existauxil} and a fixed point argument in the space $Y$, in the spirit of \cite[Theorem 7.1]{DdPDV2015}. We highlight the differences.
 \begin{proof}[\textbf{Proof of Proposition \ref{exist}.}]
Let $A[h]$ be the unique solution to the problem \eqref{auxiliaryp} for any $h\in L^2(\mathbb{R}^{n})$ with $\|h\|_*<\infty$ provided by Lemma \ref{existauxil}. Then $A$ is well defined and
\[\|A[h]\|_*\leq C\|h\|_*,\]
for some $C>0$ independent of $q$. Hence, to find a solution of \eqref{linearp} is equivalent to solve
\begin{equation}\label{Ap}
  \phi-A[-pW_q^{p-1}\phi]=A[g] \quad \phi\in Y.
\end{equation}

We claim that
$\mathcal{B}[\phi]:=A[-pW_q^{p-1}\phi]$ is a compact operator in $Y$.
Since $\mathcal{B}[\phi]$ is the solution to \eqref{auxiliaryp} with $h=-pW_q^{p-1}\phi$, then
\[\|\mathcal{B}[\phi]\|_*\leq C\|pW_q^{p-1}\phi\|_*\leq C\|\phi\|_*,\]
due to the boundedness of $W_q$, and thus $\mathcal{B}[\phi]\in Y$ for every $\phi\in Y$.

Let $\phi_m$ be a bounded sequence in $Y.$ By Lemma \ref{holderest}, Lemma \ref{Zij} and Lemma \ref{cij},
\begin{align*}
  \sup\limits_{x\neq y} &\frac{|\mathcal{B}[\phi_m](x)-\mathcal{B}[\phi_m](y)|}{|x-y|^s} \leq  C\Bigl(\bigl\|pW_q^{p-1}\phi_m+\sum\limits_{ij}c_{ij}^mZ_{ij}^m\bigl\|_{L^\infty(\mathbb{R}^{n})}+
  \bigl\|pW_q^{p-1}\phi_m+\sum\limits_{ij}c_{ij}^mZ_{ij}^m\bigl\|_{L^2(\mathbb{R}^{n})}\Bigl)\\
  &\leq C\Bigl(\|\phi_m\|_{L^\infty(\mathbb{R}^{n})}+\sum\limits_{ij}|c_{ij}^m|\|Z_{ij}^m\|_{L^\infty(\mathbb{R}^{n})}
  +\|\phi_m\|_*\|W_q^{p-1}\rho_q\|_{L^2(\mathbb{R}^{n})}+\sum\limits_{ij}|c_{ij}^m|\|Z_{ij}^m\|_{L^2(\mathbb{R}^{n})}\Bigl)\\
  &\leq C\Big(\|\phi_m\|_*+\sum\limits_{ij}|c_{ij}^m|\Big)\leq C.
\end{align*}
Then $\mathcal{B}[\phi_m]$ is equicontinuous and it converges to a function $\tilde{b}$ uniformly on a compact set. Namely,
\begin{equation}\label{Blimit}
  \|\mathcal{B}[\phi_m]-\tilde{b}\|_{L^{\infty}(\cup_{i=1}^k  B_R(q_i))}\rightarrow 0\quad \mbox{ when }m\to \infty,\qquad \mbox{ for any }R>0.
\end{equation}
If $x\in \mathbb{R}^{n}\backslash \bigcup\limits_{i=1}^k B_R(q_i),$ then
\begin{align*}
|W_q^{p-1}\phi_m| &\leq \|\phi_m\|_*|W_q^{p-1}\rho_q| \leq C \|\phi_m\|_*\bigg|\bigg(\sum\limits_{i=1}^k w_i^2\bigg)^{\frac{p-1}{2}}\rho_q\bigg|\leq C  \|\phi_m\|_*\rho_q^{1+\frac{p-1}{2}},
\end{align*}
and hence
\[\sup\limits_{x\in \mathbb{R}^{n}\backslash \cup_{i=1}^k B_R(q_i)}|\rho_q^{-1}\mathcal{B}[\phi_m]|\leq C\sup\limits_{x\in \mathbb{R}^{n}\backslash \cup_{i=1}^k B_R(q_i)}\rho_q^{\frac{p-1}{2}},\]
since $\|\phi_m\|_*$ is bounded. It follows that
\[\sup\limits_{x\in \mathbb{R}^{n}\backslash \cup_{i=1}^k B_R(q_i)}|\rho_q^{-1}\tilde{b}|\leq C\sup\limits_{x\in \mathbb{R}^{n}\backslash\cup_{i=1}^k  B_R(q_i)}\rho_q^{\frac{p-1}{2}},\]
and
\begin{align*}
 \|\mathcal{B}[\phi_m]-\tilde{b}\|_* &\leq \sup\limits_{x\in \cup_{i=1}^k  B_R(q_i)}|\rho_q^{-1}(\mathcal{B}[\phi_m]-\tilde{b})| + C\sup\limits_{x\in \mathbb{R}^{n}\backslash \cup_{i=1}^k  B_R(q_i)}\rho_q^{\frac{p-1}{2}}\\
  &\leq \frac{1}{k}(1+R)^\mu\|(\mathcal{B}[\phi_m]-\tilde{b})\|_{L^\infty(\cup_{i=1}^k  B_R(q_i))}+ C\sup\limits_{x\in \mathbb{R}^{n}\backslash \cup_{i=1}^k  B_R(q_i)}\rho_q^{\frac{p-1}{2}}.
\end{align*}
Therefore, by \eqref{Blimit}, $ \|\mathcal{B}[\phi_m]-\tilde{b}\|_*\to 0$ when $m\to\infty$, so $\mathcal{B}$ is a compact operator in $Y$.

Furthermore, from Lemma \ref{existauxil} it follows that $\phi=A[g]=0$ is the only solution to \eqref{auxiliaryp} if $g=0$. Hence, by the Fredholm's alternative, there exists a unique solution $\phi$ to \eqref{Ap} for any $g\in Y$. This and Lemma \ref{priestimate} prove Proposition \ref{exist}.
\end{proof}

We end this section by analyzing the differentiability of the solution $\phi$ of \eqref{linearp} with respect to the parameter $q.$ For this, we define the operator $T_q$ that associates  any $g\in L^2(\mathbb{R}^{n})$ with $\|g\|_*<\infty$ with the corresponding solution of \eqref{linearp}; that is,
\begin{equation}\label{defTg}
\phi:=T_q[g]\,\, \mbox{is the unique solution of \eqref{linearp} in} \, Y.
\end{equation}
Thanks to Proposition \ref{exist}, the operator $T_q$ is linear and continuous from $Y$ to $Y$.% and we write as $T_q\in L(Y).$

\begin{proposition}\label{differT}
The map $q\mapsto T_q$ is continuously differentiable in $\Xi_\eta$. Furthermore, there exists $C>0$, independent of $q$, such that
\begin{equation}\label{eq:derT}
\Bigl\|\frac{\partial T_q}{\partial q_{ij}}\Bigl\|_*\leq C\Big(\|g\|_*+\Bigl\|\frac{\partial g}{\partial q_{ij}}\Bigl\|_*\Big),\qquad \mbox{ for every }\; i\in\{1,\ldots, k\},\;j\in\{1,\ldots,n\}.
\end{equation}
\end{proposition}
\begin{proof}
Denote $q=(q_1,\ldots,q_k)$, $q_i=(q_{i1},\ldots,q_{in})$. Fix $i\in \{1,\ldots, k\}$ and $j\in \{1,\ldots, n\}$. Let us define  $q_{ij}^t:=q_i+te_j$ with $e_j$ the $j-$th element of the canonical basis, and let $$q^t:=(q_1,\ldots,q_{i-1}, q_{ij}^t,\ldots,q_k).$$
For a function $f(q)$, define
\[D^t f:=\frac{f(q^t)-f(q)}{t}.\]
By simplicity of notation, we omit the dependence on $i,j$, which are fixed.
We also set
\[\psi^t:= D^tT_{q}[g],	\qquad c_{\ell m}^t:=\alpha_\ell^{-1}\int_{\R^n}\psi^tZ_{\ell m}\,dx,\]
with $\alpha_\ell$ given in Lemma \ref{orthogonal}. Let $\tau$ be a smooth radial function such that $\tau_i(x):=\tau(|x-q_i|)\in C_0^\infty(\Omega_\varepsilon)$ for every $i\in \{1,\ldots, k\}$. Then, proceeding like in \cite[Proof of Proposition 7.5]{DdPDV2015}, the function
\[\tilde{\psi}^t:=\psi^t-\sum\limits_{\ell,m}c_{\ell m}^t \tau_\ell Z_{\ell m},\]
satisfies 
\[(-\Delta)^s\tilde{\psi}^t+V(\varepsilon x)\tilde{\psi}^t-pW_q^{p-1}\tilde{\psi}^t=\tilde{g}+\sum\limits_{\ell, m}d_{\ell m}^tZ_{\ell m},\]
where
$$d_{\ell m}^t:= D^t c_{\ell m},\qquad \tilde{g}:=p(D^tW_q^{p-1})\phi-D^tg+\sum\limits_{\ell, m}c_{\ell m}D^t Z_{\ell m}- \sum\limits_{\ell, m}[(-\Delta)^s+V(\varepsilon x)-pW_q^{p-1}] c_{\ell m}^t \tau_\ell Z_{\ell m}.$$
Notice that $\tilde{\psi}^t\in H^s(\R^n)$, $\tilde{\psi}^t=0$ in $\R^n\setminus \Omega_\varepsilon$ and $\int_{\Omega_\varepsilon}\tilde{\psi}^t Z_{\ell m}\,dx=0$ for every $\ell\in \{1,\ldots,k\}$, $m\in\{1,\ldots,n\}$. Thus, by Lemma \ref{existauxil},
$$\|\tilde{\psi}^t\|_*\leq C\|\tilde{g}\|_*.$$
Using the decay properties of $w$ and Lemma \ref{Zij}, it is easy to check that
\begin{equation*}\begin{split}
|D^tW_q^{p-1}(x)|&=\bigg|\frac{1}{t}\int_0^t\frac{d}{d\eta}\bigg(\sum_{\ell\neq i}w_\ell(x)+w_{\lambda_i}(x-q_i-\eta e_j)\bigg)^{p-1}\,d\eta\bigg|\\
&\leq \frac{p-1}{t}\int_0^t\bigg|\sum_{\ell\neq i}w_\ell(x)+w_{\lambda_i}(x-q_i-\eta e_j)\bigg|^{p-2}|\nabla w_{\lambda_i}(x-q_i-\eta e_j)\bigg|\,d\eta \leq C.
\end{split}\end{equation*}
Hence, by Lemmata \ref{graZij}, \ref{cij} and \ref{priestimate}, we conclude
$$\|\tilde{\psi}^t\|_*\leq C(\|g\|_*+\|D^tg\|_*),$$
for some $C>0$ independent of $q$ and, as a consequence,
$$\|\psi^t\|_*\leq C(\|g\|_*+\|D^tg\|_*).$$
Sending $t\to0$ we get \eqref{eq:derT}. 

The fact that the map $q\mapsto T_q$ is continuously differentiable in $\Xi_\eta$ follows by an application of the implicit function theorem together with the previous estimates.
\end{proof}

\section{The non linear projected problem}
\label{sec:nonlinear}
In this section, we focus on the non linear projected problem
\begin{equation}\label{projectedp}
\left\{\begin{array} {ll}
 (-\Delta)^s\phi+V(\varepsilon x)\phi-pW_q^{p-1}\phi=E(\phi)+N(\phi)+\sum\limits_{i=1}^k\sum\limits_{j=1}^n c_{ij}Z_{ij} & \mbox{in }\; \Omega_\varepsilon,\\
 \phi=0 & \mbox{in }\; \mathbb{R}^{n}\setminus\Omega_\varepsilon,\\
 \int_{\Omega_\varepsilon} \phi Z_{ij}=0 & i=1,\ldots,k,\;\,j=1,\ldots, n,
\end{array}\right.
\end{equation}
where the functions $E(\phi),N(\phi)$ are defined in \eqref{E} and \eqref{N}.
The main result of this section is the following:
\begin{proposition} \label{Phi}
Let $q\in \Xi_\eta$. If $\varepsilon$ is small enough, there exists a unique solution $\phi\in H^s(\mathbb{R}^n)$ to the equation \eqref{projectedp},  for certain coefficients $c_{ij}$, and a positive constant $C_0$ such that
\[\|\phi\|_*\leq C_0 \big(\varepsilon+\eta^{\mu-n-2s}\big),\]
with $\eta$ and $\mu$ specified in \eqref{Xi} and \eqref{mu} respectively.
\end{proposition}

In order to prove Proposition \ref{Phi}, we need some estimates on $N(\phi)$ and $E(\phi)$. We will use the next auxiliary lemma:
\begin{lemma} \label{ineq}
Denote $t^r:=|t|^{r-1}t$, $t\in \R$. For any $a, b\in \R$,  $r>0$, there exists a positive constant, depending only on $r$, such that
 \begin{equation}\label{point1}
	 |(a+b)^r-a^r|\leq 
	 \begin{cases}C|a|^{r-1} |b|,\qquad \mbox{if }|b|\leq|a|,\\
	 C|b|^r,\qquad \mbox{if }|a|\leq|b|.
	 \end{cases}
	\end{equation}	
Furthermore, if $r>1$,
 \begin{equation*}\label{point2}
	 |(a+b)^r-a^r-ra^{r-1}b|\leq 
	 \begin{cases}C|a|^{r-2} |b|^2,\qquad \mbox{if }|b|\leq|a|,\\
	 C|b|^r,\qquad \mbox{if }|a|\leq|b|.
	 \end{cases}
	\end{equation*}	
\end{lemma}
\begin{proof}
Let us first prove \eqref{point1}. If $a=0$ the inequalities trivially follow. %If $a\neq0,$ but $|b|=|a|,$ the inequality follows at once.
For $a\neq0,$ we can write
\begin{align*} |(a+b)^r-a^r|&=|a|^r\bigg|\bigg(1+\frac{b}{a}\bigg)^r-1\bigg|.
\end{align*}
If $r\geq 1$, by the mean value theorem,
\begin{align*} |a|^r\left|\left(1+\frac{b}{a}\right)^r-1\right|\leq r|a|^r\bigg(1+\bigg|\frac{b}{a}\bigg|\bigg)^{r-1}\left|\frac{b}{a}\right|\leq 2^{r-1}r|a|^{r}\bigg(1+\bigg|\frac{b}{a}\bigg|^{r-1}\bigg)\left|\frac{b}{a}\right|,
\end{align*}
and the result follows.

Suppose $0<r<1$ and define $f(x):=\left(1+x\right)^r-1$. 

\noindent If $|x|\geq1$, then
\[|f(x)|=|(1+x)^r-1|\leq 1+(1+|x|)^r\leq 1+2^r|x|^r,\]
and hence
$$|a|^r\bigg|\bigg(1+\frac{b}{a}\bigg)^r-1\bigg|\leq |a|^r+2^r|b|^r\leq (1+2^r)|b|^r\qquad \mbox{whenever }|a|\leq |b|.$$
If $0\leq x<1$, since $r-1<0$, by the mean value theorem $|f(x)|\leq r|x|$.
If $-1<x<0,$\begin{gather*}
              f(x)<0, \quad f'(x)=r(1+x)^{r-1}\geq0,\quad f''(x)=r(r-1)(1+x)^{r-2}\leq0.
            \end{gather*}
Then $f$ is negative, increasing and concave and $f(-1)=-1$, and hence $|f(x)|\leq |x|$.
Consequently,
\begin{equation*}\label{ineq2}
  |a|^r\bigg|\bigg(1+\frac{b}{a}\bigg)^r-1\bigg|=|a|^r\bigg|f\bigg(\frac{b}{a}\bigg)\bigg|\leq |a|^{r-1}|b|\qquad \mbox{whenever }|b|\leq |a|.
\end{equation*}
Inequality \eqref{point2} follows observing that
 $$ (a+b)^r-a^r-ra^{r-1}b=r\int_0^1\left[(a+\sigma b)^{r-1}-a^{r-1}\right]b\,d\sigma,$$
 and applying \eqref{point1}.
\end{proof}
We can now estimate the non linear term. 
\begin{lemma} \label{estimateN} Let $\phi\in Y$ and $q\in\Xi_\eta$. Then, there exists a constant $C>0$ independent of $q$ such that
\[\|N(\phi)\|_*\leq C(\|\phi\|_*^2+\|\phi\|_*^p).\]
\end{lemma}
\begin{proof} By Lemma \ref{ineq} and the fact that $\rho_q$ is bounded (independently of $q$),
\begin{align*}
 \rho_q^{-1}|N(\phi)|\leq C \rho_q^{-1}(|\phi|^{2}+|\phi|^p)\leq C ( \rho_q\rho_q^{-2}|\phi|^{2}+\rho_q^{p-1} \rho_q^{-p}|\phi|^p)\leq C(\|\phi\|_*^{2}+\|\phi\|_*^p),
 \end{align*}
 and the estimate follows.
\end{proof}

\begin{lemma} \label{limit}Let $q\in\Xi_\eta$. There exists a positive constant $C$, independent of $q$, such that
\begin{equation}\label{eq:estvi}
|\bar{u}_i-w_i|\leq C \varepsilon^{n+2s}\qquad \mbox{ for every }i=1,\ldots, k.
\end{equation}
Consequently,
\[|U_q-W_q|\leq C \varepsilon^{n+2s},\]
with $U_q$ and $W_q$ defined in \eqref{Wq}.
\end{lemma}
\begin{proof}
Let $v_i:=\bar{u}_i-w_i$. Since $\bar{u}_i$ and $w_i$ satisfy \eqref{dirichlet2}  and \eqref{eqwlambda} respectively, then $v_i$ solves 
\begin{equation}\label{eqvi}
\left\{\begin{array} {ll}
 (-\Delta)^sv_i+\lambda_iv_i=0 & \mbox{in }\; \Omega_\varepsilon,\\
 v_i= -w_i & \mbox{in }\; \mathbb{R}^{n}\setminus\Omega_\varepsilon.
\end{array}\right.
\end{equation}
By \eqref{behavwi}, 
\begin{equation*}
|v_i|=|w_i|\leq \frac{C}{|x-q_i|^{n+2s}}\leq C\varepsilon^{n+2s}\qquad \mbox{ in }\mathbb{R}^{n}\setminus\Omega_\varepsilon,
\end{equation*}
and thus, by the maximum principle, $|v_i|\leq C\varepsilon^{n+2s}$ in the whole space and \eqref{eq:estvi} follows.
Hence
$$|U_q-W_q|=\bigg|\sum\limits_{i=1}^{k}v_i\bigg|\leq\sum\limits_{i=1}^k|v_i|\leq C\varepsilon^{n+2s}.$$
\end{proof}

\begin{lemma} \label{limit1} Let $q\in \Xi_\eta$. There exists a positive constant $C$, independent of $q$, such that
\[\left|\frac{\partial \bar{u}_i}{\partial q_{ij}}-\frac{\partial w_i}{\partial q_{ij}}\right|\leq C \varepsilon^{\nu_1},\qquad \mbox{for every }\;i\in\{1,\ldots, k\},\;j\in\{1,\ldots,n\},\]
where $\nu_1:=\min\{n+2s+1,p(n+2s)\}$. As a consequence,
\[\left|\frac{\partial U_q}{\partial q_{ij}}-\frac{\partial W_q}{\partial q_{ij}}\right|\leq C \varepsilon^{\nu_1},\qquad \mbox{ for every }\; i\in\{1,\ldots, k\},\;j\in\{1,\ldots,n\}.\]
\end{lemma}
\begin{proof} Let $v_i=\bar{u}_i-w_i,$ which satisfies \eqref{eqvi}. Using Corollary \ref{cor:derw_qij}, it can be seen that $\frac{\partial v_i}{\partial q_{ij}}$ satisfies
\begin{equation*}
\left\{\begin{array} {ll}
 \displaystyle (-\Delta)^s\frac{\partial v_i}{\partial q_{ij}}+\lambda_i\frac{\partial v_i}{\partial q_{ij}}+\varepsilon\frac{\partial V(\xi_i)}{\partial \xi_{ij}}v_i=0 & \mbox{in }\; \Omega_\varepsilon,\\
\displaystyle \frac{\partial v_i}{\partial q_{ij}}= \frac{\partial w_i}{\partial x_j}+O\bigg(\frac{\varepsilon }{(1+|x-q_i|)^{\nu_1-1}}\bigg) & \mbox{in }\; \mathbb{R}^{n}\setminus\Omega_\varepsilon.
\end{array}\right.
\end{equation*}
Therefore, by Lemma \ref{Zij},
\begin{equation}\label{eq:esti_dervi}
\left|\frac{\partial v_i}{\partial q_{ij}}\right|\leq C\bigg(\varepsilon^{\nu_1}+O\bigg(\frac{\varepsilon }{(1+|x-q_i|)^{\nu_1-1}}\bigg)\bigg)\qquad \mbox{in }\; \mathbb{R}^{n}\setminus\Omega_\varepsilon.
\end{equation}
Let us write $\dfrac{\partial v_i}{\partial q_{ij}}=f+h$ where $f$ and $h$ solve
 \begin{equation*}
\left\{\begin{array} {ll}
 (-\Delta)^s f+\lambda_i f=0 & \mbox{in }\; \Omega_\varepsilon,\\
\displaystyle f= \dfrac{\partial w_i}{\partial x_j}+O\bigg(\frac{\varepsilon}{(1+|x-q_i|)^{\nu_1-1}}\bigg) & \mbox{in }\; \mathbb{R}^{n}\setminus\Omega_\varepsilon,
\end{array}\right.
\qquad 
\left\{\begin{array} {ll}
 (-\Delta)^s h+\lambda_i h=-\varepsilon\dfrac{\partial V(\xi_i)}{\partial \xi_{ij}}v_i & \mbox{in }\; \Omega_\varepsilon,\\
h= 0 & \mbox{in }\; \mathbb{R}^{n}\setminus\Omega_\varepsilon.
\end{array}\right.
\end{equation*}
Notice that 
$$|x-q_i|\geq d\geq \frac{\delta_*}{\varepsilon}\quad \mbox{ for every }x\in \R^n\setminus \Omega_\varepsilon,$$  with $d$ defined in \eqref{defD}. Hence, using Lemma \ref{Zij} and the maximum principle,
\[\left|f \right|\leq C\varepsilon^{\nu_1}\qquad \mbox{in }\; \mathbb{R}^{n}.\]
Furthermore, by standard elliptic regularity estimates together with \eqref{eq:estvi}, 
\[\left\|h\right\|_{L^{\infty}(\Omega_\varepsilon)}\leq C \sup\limits_{\Omega_\varepsilon}\bigg|\varepsilon\frac{\partial V(\xi_i)}{\partial \xi_{ij}}v_i\bigg|\leq C \varepsilon^{n+2s+1}|\nabla V(\xi_i)|.\]
Hence, applying Remark \ref{unifV}, \eqref{eq:esti_dervi} holds in $\R^n$ and
\[\left|\frac{\partial U_q}{\partial q_{ij}}-\frac{\partial W_q}{\partial q_{ij}}\right|=\left|\frac{\partial v_i}{\partial q_{ij}}\right|\leq C\varepsilon^{\nu_1}.\]
\end{proof}

Let us estimate the error $E$, given in \eqref{E}, for small functions in the space $Y$ (recall its definition in \eqref{defY}).

\begin{lemma} \label{estimateE}
Let $\phi\in Y$ with $\|\phi\|_*\leq1$ and $q\in\Xi_\eta$. Then there exists $C>0$, independent of $q$, such that
\[\|E(\phi)\|_*\leq C\big(\varepsilon+\eta^{\mu-n-2s}\big),\]
with $\eta$ and $\mu$ given in \eqref{Xi} and \eqref{mu} respectively.
\end{lemma}

\begin{proof}
By Lemmas \ref{ineq} and \ref{limit},
\begin{align*}
  |(U_q+\phi)^p-(W_q+\phi)^p|\leq C\left(|U_q-W_q||W_q+\phi|^{p-1}+|U_q-W_q|^p\right)\leq C \left(\varepsilon^{n+2s}|W_q+\phi|^{p-1}+\varepsilon^{(n+2s)p}\right),
\end{align*}
and thus, since $\|\phi\|_*\leq 1$ and $\|W_q\|_*$ is uniformly bounded in $q$,
\begin{equation}\label{eq:E1}
\|(U_q+\phi)^p-(W_q+\phi)^p\|_*\leq C\varepsilon^{n+2s}.
\end{equation}
On the other hand, by Remark \ref{unifV},
\begin{align*}
  \sup\limits_{x\in \R^n}\Big|\rho_q(x)^{-1}\sum\limits_{i=1}^k(\lambda_i-V(\varepsilon x))\bar{u}_i\Big|&=\sup\limits_{x\in \Omega_\varepsilon}\Big|\rho_q(x)^{-1}\sum\limits_{i=1}^k(\lambda_i-V(\varepsilon x))\bar{u}_i\Big|\\
  &\leq C \sup\limits_{x\in \Omega_\varepsilon}\bigg(\rho_q(x)^{-1}\sum\limits_{i=1}^k\varepsilon |x-q_i||w_i+\varepsilon^{n+2s}|\bigg).
\end{align*}
Furthermore, 
$$\varepsilon |x-q_i|\leq \mbox{diam}(\Omega)<+\infty\qquad \mbox{for every }i=1,\ldots, k\mbox{ and }x\in\Omega_\varepsilon.$$
Combining this estimate with \eqref{behavwi} we get
\begin{align*}
  \rho_q(x)^{-1}\sum\limits_{i=1}^k&\varepsilon|x-q_i||w_i+\varepsilon^{n+2s}| \leq C\sum\limits_{i=1}^k\varepsilon|x-q_i|\bigg(\frac{1}{(1+|x-q_i|)^{n+2s-\mu}}+\varepsilon^{n+2s-\mu}\bigg)\leq C\varepsilon,
\end{align*}
since $n+2s-\mu-1\geq 0$.
Therefore
\begin{equation}\label{eq:E2}
\Big\|\sum\limits_{i=1}^k(\lambda_i-V(\varepsilon x))\bar{u}_i\Big\|_*\leq C \varepsilon.
\end{equation}
Fix $\ell=1,\ldots, k$ and consider the subdomain 
\begin{equation}\label{Omegai}
  \Omega_\ell=\{x\in \Omega_\varepsilon: w_\ell(x)\geq w_i(x),\;\; \forall \,\ell\neq i\}\qquad \ell=1,\ldots,k.
\end{equation}
In this region,
\begin{equation}\begin{split}\label{eq:E3}
 \rho_q^{-1}\left|\bigg(\sum\limits_{i=1}^kw_i\bigg)^{p}-\sum\limits_{i=1}^kw_i^{p}\right|&\leq C \rho_q^{-1}w_\ell^{p-1}\sum\limits_{i\neq \ell}\frac{1}{|x-q_i|^{n+2s}}\leq  C\sum\limits_{i\neq \ell}\frac{1}{|x-q_i|^{n+2s-\mu}}\\
   &\leq  C\sum\limits_{i\neq \ell}\frac{1}{|q_\ell-q_i|^{n+2s-\mu}}\leq C \eta^{\mu-n-2s}.
\end{split}\end{equation}
Putting together \eqref{eq:E1}, \eqref{eq:E2} and \eqref{eq:E3} for every $i$ we conclude the result.
\end{proof}

We can already prove Proposition \ref{Phi}. To do so, we will adapt the strategy of \cite[Theorem 7.6]{DdPDV2015}.
\medskip

\noindent \textit{Proof of Proposition \ref{Phi}.} Let $T_q$ defined in \eqref{defTg}. We want to prove the existence of $\phi$ such that
\[\phi=T_q[E(\phi)+N(\phi)].\]
Define
\[K_q(\phi):=T_q[E(\phi)+N(\phi)].\]
Given $\varepsilon$ small enough and $C_0>0$ to be chosen later, we define the set
\[B:=\{\phi\in Y: \|\phi\|_*\leq C_0 \tau\},\qquad \tau:=\varepsilon+\eta^{\mu-n-2s}.\]
We claim that
\begin{equation}\label{Kq}
  K_q \, \mbox{is a contraction mapping from $B$ into $B$}.
\end{equation}
Let us see first that $K_q(\phi)\in B$ provided  $\phi\in B$. Indeed, if $\phi\in B$, applying Proposition \ref{exist}, Lemma \ref{estimateN} and Lemma \ref{estimateE} we get
 \begin{align*}
 \|K_q(\phi)\|_*& \leq C\|E(\phi)+N(\phi)\|_*\leq  C\left(\|E(\phi)\|_* +C_1(\|\phi\|_*^2+\|\phi\|_*^p) \right) \leq  C (C_2\tau+ C_1C_0^2\tau^2+  C_1C_0^p\tau^p)
 \\
  & = C_0\tau\left(\frac{C C_2}{C_0}+C C_1C_0\tau+C C_1C_0^{p-1}\tau^{p-1}\right).
   \end{align*}
Choosing
\begin{equation*}
C_0>2C C_2,\qquad 
		\tau<\tau_1=
		\begin{cases}
			\frac{1}{2CC_1(C_0+C_0^{p-1})} &\mbox{if}\,\,p\geq2,\\
			\left(\frac{1}{2CC_1(C_0+C_0^{p-1})}\right)^\frac{1}{p-1} &\mbox{if}\,\,1<p<2,
		\end{cases}
	\end{equation*}
we deduce $\|K_q(\phi)\|_*\leq C_0\tau$ and then $K_q(\phi)\in B$.

To see that the application is contractive, assume $\phi_1,\phi_2$. Then we can write
\begin{equation*}
|N(\phi_1)-N(\phi_2)|\leq |(W_q+\phi_1)^p- (W_q+\phi_2)^p-p(W_q+\phi_2)^{p-1}(\phi_1-\phi_2)|+p|(W_q+\phi_2)^{p-1}-W_q^{p-1}||\phi_1-\phi_2|,
\end{equation*}
and hence, applying Lemma \ref{ineq} to the first term in the right hand side, and \cite[Lemma 7.11]{DdPDV2015} to the second one,
\begin{equation*}
|N(\phi_1)-N(\phi_2)|\leq C(|\phi_1|^{p-1}+|\phi_2|^{p-1}+|\phi_1|+|\phi_2|)|\phi_1-\phi_2|,
\end{equation*}
with $C$ independent of $q$. Thus,
\begin{align*}
  \|N(\phi_1)-N(\phi_2)\|_* & \leq C(\|\phi_1\|_*^{p-1}+\|\phi_2\|_*^{p-1}+\|\phi_1\|_*+\|\phi_2\|_*)\|\phi_1-\phi_2\|_*\\
  &\leq C(C_0+C_0^{p-1})\tau^{\min\{1,p-1\}}\|\phi_1-\phi_2\|_*.
\end{align*}
Fix $x\in \Omega_\varepsilon$. Given $t$ in a bounded subset of $\mathbb{R},$ we consider the function
\[f(t):=(U_q(x)+t)^p-(W_q(x)+t)^p.\]
Using \cite[Lemma 7.11]{DdPDV2015},
\[|f'(t)|=p|(U_q+t)^{p-1}-(W_q+t)^{p-1}|\leq C|U_q-W_q|^{\min\{1,p-1\}},\]
and hence 
\[|E(\phi_1)-E(\phi_2)|\leq C|U_q-W_q|^{\min\{1,p-1\}}|\phi_1-\phi_2|,\]
where we have used the fact that $\phi_1,\phi_2$ are bounded since they belong to $B$.
Thanks to Lemma \ref{limit}, we thus have 
\[\|E(\phi_1)-E(\phi_2)\|_*\leq C\varepsilon^{\min\{1,p-1\}(n+2s)}\|\phi_1-\phi_2\|_*.\]
Therefore, using the estimates above and Proposition \ref{exist},
\begin{align*}
  \|K_q(\phi_1)-K_q(\phi_2)\|_*&\leq C(\|E(\phi_1)-E(\phi_2)\|_*+\|N(\phi_1)-N(\phi_2)\|_*) \\
  &\leq C(\varepsilon^{\min\{1,p-1\}(n+2s)}+\tau^{\min\{1,p-1\}})\|\phi_1-\phi_2\|_*.
\end{align*}
Choosing $\varepsilon$ (and hence $\tau$) small enough,
\[\|K_q(\phi_1)-K_q(\phi_2)\|_*<\|\phi_1-\phi_2\|_*,\]
which completes the proof of \eqref{Kq}.

By a fixed point argument, there exists a unique solution $\phi\in B$ to \eqref{projectedp}. 
$\hfill\square$\\

We now estimate the derivative of the error term $E(\phi).$
\begin{lemma} \label{estimateE1}
Let $q\in\Xi_\eta$ and $\phi\in Y$ with $\|\phi\|_*\leq {C_0}\tau$, where $\tau:=\varepsilon+\eta^{\mu-n-2s}$ and $C_0$ given in Proposition \ref{Phi}. Then, there exists a positive constant $C$, independent of $q$, such that
\[\left\|\frac{\partial E(\phi)}{\partial q_{ij}}\right\|_*\leq C\varepsilon^{\min\{1,(p-1)(n+2s)\}}, \qquad \mbox{ for every }\; i\in\{1,\ldots, k\},\;j\in\{1,\ldots,n\}.\]
\end{lemma}
\begin{proof}
By Proposition \ref{differT}, the function $\frac{\partial\phi}{\partial q_{ij}}$ is well defined, and we can write
\begin{align*}
  \frac{\partial E(\phi)}{\partial q_{ij}} &=I_1+I_2+I_3,
\end{align*}
where
\begin{align*}
  I_1&:=p[(U_q+\phi)^{p-1}-(W_q+\phi)^{p-1}]\left(\frac{\partial \phi}{\partial q_{ij}}+\frac{\partial W_q}{\partial q_{ij}}\right)+p(U_q+\phi)^{p-1}\left(\frac{\partial U_q}{\partial q_{ij}}-\frac{\partial W_q}{\partial q_{ij}}\right),
  \end{align*}
  \begin{align*}
 I_2&:=(\lambda_i-V(\varepsilon x))\frac{\partial \bar{u}_i}{\partial q_{ij}}+\varepsilon \frac{\partial V(\xi_i)}{\partial \xi_{ij}} \bar{u}_i,\qquad 
  I_3:=p\bigg(\bigg(\sum\limits_{\ell=1}^kw_\ell\bigg)^{p-1}-w_i^{p-1}\bigg)\frac{\partial w_i}{\partial q_{ij}}.
\end{align*}
Let us estimate $I_1$. By \cite[Lemma 7.11]{DdPDV2015}, Lemma \ref{limit} and Lemma \ref{limit1},
\begin{equation*}\begin{split}
|I_1|&\leq C\bigg(|U_q-W_q|^r\bigg(\bigg|\frac{\partial\phi}{\partial q_{ij}}\bigg|+\bigg|\frac{\partial W_q}{\partial q_{ij}}\bigg|\bigg)+(|U_q|^{p-1}+|\phi|^{p-1})\bigg|\frac{\partial U_q}{\partial q_{ij}}-\frac{\partial W_q}{\partial q_{ij}}\bigg|\bigg)\\
&\leq C\bigg(\varepsilon^{r(n+2s)}\bigg(\bigg|\frac{\partial\phi}{\partial q_{ij}}\bigg|+\bigg|\frac{\partial W_q}{\partial q_{ij}}\bigg|\bigg)+\varepsilon^{\nu_1}(|U_q|^{p-1}+|\phi|^{p-1})\bigg)
\end{split}\end{equation*}
where $r:=\min\{1,p-1\}$. Notice that,
\begin{equation*}\begin{split}
\sup_{x\in\R^n} &\big(\rho_q^{-1} \varepsilon^{\nu_1}|U_q|^{p-1}\big)=\sup_{x\in\Omega_\varepsilon} \big(\rho_q^{-1} \varepsilon^{\nu_1}\big(\varepsilon^{(p-1)(n+2s)}+|W_q|^{p-1}\big)\big)\\
&\leq C\big(\varepsilon^{\nu_1-\mu+(p-1)(n+2s)}+\|W_q\|_*^{p-1}\varepsilon^{\nu_1-\mu (2-p)_+}\big)\leq C\varepsilon,
\end{split}\end{equation*}
where in the last step we used \eqref{mu} and  the fact that $\nu_1=\{n+2s+1,p(n+2s)\}.$
Likewise,
\begin{equation*}\begin{split}
\sup_{x\in\R^n} &\big(\rho_q^{-1} \varepsilon^{\nu_1}|\phi|^{p-1}\big)\leq C\big(\|\phi\|_*^{p-1}\varepsilon^{\nu_1-\mu (p-2)_+}\big)\leq C\big(\tau^{p-1}\varepsilon^{\nu_1-\mu (p-2)_+}\big)\leq C\tau^{p-1}\varepsilon.
\end{split}\end{equation*}
Hence, by virtue of Lemma \ref{Zij} and  Corollary \ref{cor:derw_qij},
\[\|I_1\|_*\leq C\left(\varepsilon^{r(n+2s)}\left(\left\|\frac{\partial \phi}{\partial q_{ij}}\right\|_*+\left\|\frac{\partial W_q}{\partial q_{ij}}\right\|_*\right)+\varepsilon\right)\leq C\left(\varepsilon^{r(n+2s)}\left\|\frac{\partial \phi}{\partial q_{ij}}\right\|_*+\varepsilon^{\min\{1,(p-1)(n+2s)\}}\right)\]
where $r:=\min\{1,p-1\}$.  

Using Lemma \ref{limit1}, Lemma \ref{Zij}, and the boundedness of $\Omega$ and Remark \ref{unifV}, 
\begin{align*}
  |I_2|
  &\leq C \varepsilon\left(|x-q_i|\Big|\frac{\partial w_i}{\partial q_{ij}}+\varepsilon^{\nu_1}\Big|+ |w_i+\varepsilon^{n+2s}|\right)\\
  &\leq C \rho_q\varepsilon\left(\frac{|x-q_i|}{(1+|x-q_i|)^{\nu_1-\mu}}+(1+|x-q_i|)^{\mu+1}\varepsilon^{\nu_1}+\frac{1}{(1+|x-q_i|)^{n+2s-\mu}}+(1+|x-q_i|)^{\mu}\varepsilon^{n+2s}\right)\\
  &\leq C \rho_q\varepsilon\big(1+\varepsilon^{\nu_1-\mu-1}+\varepsilon^{n+2s-\mu}\big)\leq C \rho_q\varepsilon.
\end{align*}
Let us finally estimate $I_3$. Consider, for every $i=1,\ldots, k$, the subdomain $\Omega_i$ defined in \eqref{Omegai}.
In this region
$$|x-q_\ell|\geq \frac{|q_i-q_\ell|}{2}\quad \mbox{ for every }\ell\neq i,$$
and thus, by Corollary \ref{cor:derw_qij}, Lemma \ref{Zij}, Lemma \ref{ineq} and \eqref{behavwi} we have
\begin{align*}
\bigg|p\bigg(\bigg(\sum\limits_{\ell=1}^kw_\ell\bigg)^{p-1}-w_i^{p-1}\bigg)\frac{\partial w_i}{\partial q_{ij}}\bigg| & \leq C(1+|x-q_i|)^{-\nu_1}\left(w_i^{p-2}\sum\limits_{\ell \neq i}w_\ell+\bigg(\sum\limits_{\ell\neq i}w_\ell\bigg)^{p-1}\right)\\
&\leq \frac{C}{(1+|x-q_i|)^{(n+2s)(p-2)+\nu_1+\mu}}\sum\limits_{\ell\neq i}\frac{1}{|x-q_\ell|^{n+2s-\mu}}\\
&\leq \frac{C}{(1+|x-q_i|)^{(n+2s)(p-2)+\nu_1+\mu}}\sum\limits_{\ell\neq i}\frac{1}{|q_i-q_\ell|^{n+2s-\mu}}\\
&\leq C\rho_q\eta^{\mu-n-2s}
\end{align*}
with $C>0$ independent of $q$. Repeating the argument for every $i$ we decude $\|I_3\|_*\leq C \eta^{\mu-n-2s}$ and therefore
\begin{equation}\label{eq:derErr}
\left\|\frac{\partial E(\phi)}{\partial q_{ij}}\right\|_*\leq C\left(\varepsilon^{r(n+2s)}\left\|\frac{\partial \phi}{\partial q_{ij}}\right\|_*+\varepsilon^{\min\{1,(p-1)(n+2s)\}}+\eta^{\mu-n-2s}\right),
\end{equation}
for $r:=\min\{1,p-1\}$. By Proposition \ref{differT}, Lemma \ref{estimateN} and Lemma \ref{estimateE},
$$\left\|\frac{\partial \phi}{\partial q_{ij}}\right\|_*\leq C\left(\tau + \varepsilon^{r(n+2s)}\left\|\frac{\partial \phi}{\partial q_{ij}}\right\|_*+\varepsilon^{\min\{1,(p-1)(n+2s)\}}+\eta^{\mu-n-2s}+\left\|\frac{\partial N(\phi)}{\partial q_{ij}}\right\|_*\right).$$
Proceeding like in \cite[Lemma 7.14]{DdPDV2015} one can prove 
$$\left\|\frac{\partial N(\phi)}{\partial q_{ij}}\right\|_*\leq C\left(\tau^r\left\|\frac{\partial \phi}{\partial q_{ij}}\right\|_*+1\right),$$
and hence, for $\varepsilon$ small enough
$$\left\|\frac{\partial \phi}{\partial q_{ij}}\right\|_*\leq C,$$
with $C>0$ independent of $q$. Substituting in \eqref{eq:derErr} and using the fact that $\eta\geq \frac{\delta_\star}{\varepsilon}$ we conclude the result.
\end{proof}

\begin{proposition} \label{Phi1} If the problem \eqref{projectedp} has a unique solution $\Phi(q)$ for every $q\in\Xi_\eta$, then the map $q\mapsto \Phi(q)$ is $C^1$ and there exists a positive constant $C$, independent of $q$, such that
\[\left\|\frac{\partial\Phi(q)}{\partial q_{ij}}\right\|_*\leq C\left(\|E\|_*+\left\|\frac{\partial E(\phi)}{\partial q_{ij}}\right\|_*\right)\qquad \mbox{ for every }\;i\in\{1,\ldots, k\},\;j\in\{1,\ldots,n\}.\]
\end{proposition}
\begin{proof}
The proof follows as a consequence of Proposition \ref{differT} and the implicit function theorem reproducing the arguments in \cite[Proposition 5.1]{DdPW2014}, so we skip the details.
\end{proof}

\section{The variational reduction}
\label{energy}
Problem \eqref{dirichlet1} has a variational structure: its solutions can be seen as critical points of the functional
\begin{equation}\label{Jepsilon}
  J_\varepsilon(u):=\frac{1}{2}\int_{\Omega_\varepsilon}((-\Delta)^suu+V(\varepsilon x)u^2)\,dx-\frac{1}{p+1}\int_{\Omega_\varepsilon} u^{p+1}\,dx, \quad u\in H^s_0(\Omega_\varepsilon).
\end{equation}
Let us also consider the functionals associated to problem \eqref{eqwlambda} for every $\lambda_i$, $i=1,\ldots, k$, 
\begin{equation}\label{Ji}
J_i(u):=\frac{1}{2}\int_{\mathbb{R}^n}((-\Delta)^suu+\lambda_i u^2)\,dx-\frac{1}{p+1}\int_{\mathbb{R}^n} u^{p+1}\,dx,\qquad u\in H^s(\mathbb{R}^n).
\end{equation}
\begin{proposition}\label{propenergy} Let $q\in\Xi_\eta$, and  $U_q$, $W_q$ defined in \eqref{Wq}. Then
\begin{equation}\label{eqenergy}
  J_\varepsilon (U_q)=J(W_q)+O\bigg(\varepsilon \sum\limits_{i=1}^k |\nabla V(\xi_i)|+\varepsilon^2\bigg),
\end{equation}
where 
\begin{equation}\begin{split}\label{JWq}
J(W_q):=\sum\limits_{i=1}^kJ_i(w_i)+\frac{1}{2}\sum\limits_{i\neq \ell} \int_{\mathbb{R}^n} w_i^pw_\ell\,dx-\frac{1}{p+1}\int_{\mathbb{R}^n}\bigg(\bigg(\sum\limits_{i=1}^kw_i\bigg)^{p+1}-\sum\limits_{i=1}^k w_i^{p+1}\bigg)\,dx.
\end{split}\end{equation}
\end{proposition}
\begin{proof}By \eqref{dirichlet2} and Lemma \ref{expressionui}, we have that
\begin{align*}
               J_\varepsilon(U_q)&=\frac{1}{2}\int_{\Omega_\varepsilon}\sum\limits_{i=1}^k w_i^p\sum\limits_{\ell=1}^k\bar{u}_\ell\,dx+\frac{1}{2}\int_{\Omega_\varepsilon}\sum\limits_{i=1}^k(V(\varepsilon x)-\lambda_i)\bar{u}_i\sum\limits_{\ell=1}^k\bar{u}_\ell\,dx -\frac{1}{p+1}\int_{\Omega_\varepsilon} U_q^{p+1}\,dx\\
               &=\frac{1}{2}\sum\limits_{i,\ell=1}^k \int_{\Omega_\varepsilon} w_i^p(w_\ell-\Lambda_\ell-\Pi_\ell)\,dx+ \frac{1}{2}\sum\limits_{i,\ell=1}^k\int_{\Omega_\varepsilon}(V(\varepsilon x)-\lambda_i)\bar{u}_i\bar{u}_\ell\,dx-\frac{1}{p+1}\int_{\Omega_\varepsilon} U_q^{p+1}\,dx\\
               &=\frac{1}{2}\sum\limits_{i=1}^k \int_{\Omega_\varepsilon} w_i^{p+1}\,dx+\frac{1}{2}\sum\limits_{i\neq \ell} \int_{\Omega_\varepsilon} w_i^pw_\ell\,dx-\frac{1}{2}\sum\limits_{i,\ell=1}^k \int_{\Omega_\varepsilon}w_i^p(\Lambda_\ell+\Pi_\ell)\,dx\\
              &\quad+\frac{1}{2}\sum\limits_{i,\ell=1}^k \int_{\Omega_\varepsilon}(V(\varepsilon x)-\lambda_i)\bar{u}_i\bar{u}_\ell\,dx -\frac{1}{p+1}\int_{\Omega_\varepsilon} U_q^{p+1}\,dx.
             \end{align*}
Notice that, by \eqref{eqwlambda} and Lemma \ref{expressionui},
\begin{align*}
\frac{1}{2}\int_{\Omega_\varepsilon} w_i^{p+1}&=J_i(w_i)+\frac{1}{p+1}\int_{\mathbb{R}^n} w_i^{p+1}\,dx-\frac{1}{2}\int_{\mathbb{R}^n\setminus\Omega_\varepsilon}w_i^{p+1}\,dx,
\end{align*}
%\begin{align*}
 %  \int_{\Omega_\varepsilon}(V(\varepsilon x)-\lambda_i)\bar{u}_i^2 &= \int_{\Omega_\varepsilon}\{(V(\varepsilon x)-\lambda_i)[w_i^2+(\Lambda_i+\Pi_\varepsilon(q_i))^2\\
  % &\quad-2(\Lambda_i+\Pi_\varepsilon(x,q_i))w_i]\},
%\end{align*}
and
\begin{align*}
   \int_{\Omega_\varepsilon}(V(\varepsilon x)-\lambda_i)\bar{u}_i\bar{u}_\ell &= \int_{\Omega_\varepsilon}\{(V(\varepsilon x)-\lambda_i)[w_iw_\ell+(\Lambda_i+\Pi_i)(\Lambda_\ell+\Pi_\ell)\\
   &\quad-(\Lambda_i+\Pi_i)w_\ell-(\Lambda_\ell+\Pi_\ell)w_i]\}\,dx.
\end{align*}
Therefore,
\begin{equation}\begin{split}\label{expJUq}
 J_\varepsilon(U_q)&=J(W_q)-\frac{1}{2}\sum\limits_{i,\ell=1}^k \int_{\Omega_\varepsilon}w_i^p(\Lambda_\ell+\Pi_\ell(x))\,dx-\frac{1}{2}\sum\limits_{i\neq l} \int_{\mathbb{R}^n\setminus\Omega_\varepsilon} w_i^pw_\ell\,dx\\
               &\quad+\frac{1}{2}\sum\limits_{i,\ell=1}^k \int_{\Omega_\varepsilon}\{(V(\varepsilon x)-\lambda_i)[w_iw_\ell+(\Lambda_i+\Pi_i)(\Lambda_\ell+\Pi_\ell)-(\Lambda_i+\Pi_i)w_l-(\Lambda_\ell+\Pi_\ell)w_i]\}\,dx\nonumber\\
   &\quad+\frac{1}{p+1}\int_{\Omega_\varepsilon} \left(W_q^{p+1}- U_q^{p+1}\right)\,dx
+\int_{\mathbb{R}^n\setminus\Omega_\varepsilon}\bigg(\frac{1}{p+1}\bigg(\sum\limits_{i=1}^kw_i\bigg)^{p+1}-\frac{1}{2}\sum\limits_{i=1}^kw_i^{p+1}\bigg)\,dx
\end{split}\end{equation}
where $J(W_q)$ was given in \eqref{JWq}.
Let us estimate this energy term by term. Denote
$$J_1:=\sum\limits_{i,\ell=1}^k \int_{\Omega_\varepsilon}w_i^p(\Lambda_\ell+\Pi_\ell)\,dx.$$
%and
%\begin{align*}
%  J_1:&= \sum\limits_{i=1}^k \int_{\Omega_\varepsilon}w_i^p(\Lambda_i(x)+\Pi_\varepsilon(x,q_i)),\\
%  J_2:&=\sum\limits_{i\neq l} \int_{\mathbb{R}^n\setminus\Omega_\varepsilon} w_i^pw_l,\\
%  J_3:&=\sum\limits_{i\neq l} \int_{\Omega_\varepsilon}w_i^p(\Lambda_l(x)+\Pi_\varepsilon(x,q_l)),\\
%  J_4:&=\sum\limits_{i=1}^k \int_{\Omega_\varepsilon}\{(V(\varepsilon x)-\lambda_i)
%  [w_i^2+(\Lambda_i+\Pi_\varepsilon(x,q_i))^2-2(\Lambda_i+\Pi_\varepsilon(x,q_i))w_i]\},\\
%  J_5:&=\sum\limits_{i\neq l} \int_{\Omega_\varepsilon}\{(V(\varepsilon x)-\lambda_i)[w_iw_l+(\Lambda_i+\Pi_\varepsilon(x,q_i))(\Lambda_l+\Pi_\varepsilon(x,q_l))\\
%   &\quad-(\Lambda_i+\Pi_\varepsilon(x,q_i))w_l-(\Lambda_l+\Pi_\varepsilon(x,q_l))w_i]\},\\
%  J_6:&=\int_{\Omega_\varepsilon} \left(W_q^{p+1}- U_q^{p+1}\right),\\
%  %J_7:&=\int_{\mathbb{R}^n}\left(\left(\sum\limits_{i=1}^kw_i\right)^{p+1}-\sum\limits_{i=1}^k w_i^{p+1}\right),\\
%\mbox{and} \qquad J_7:&=\int_{\mathbb{R}^n\setminus\Omega_\varepsilon}\left(\frac{1}{p+1}\left(\sum\limits_{i=1}^kw_i\right)^{p+1}-\frac{1}{2}\sum\limits_{i=1}^kw_i^{p+1}\right).
%\end{align*}
Using Proposition \ref{prop:boundLambda} and \eqref{behavwi} we get
\begin{align*}
\int_{\Omega_\varepsilon}w_i^p\Lambda_\ell\,dx & \leq \frac{C}{d^{(n+2s)p}}\int_{\Omega_\varepsilon}w_i^p\,dx\leq \varepsilon^{(n+2s)p},
\end{align*}
since $d\geq \frac{\delta_\star}{\varepsilon}$.
Hence, applying Lemma \ref{minimum2} we conclude that
\begin{equation}\label{J1}
J_1\leq C\varepsilon^{n+2s}.
\end{equation}
Likewise, 
\begin{align}
  J_2& :=\sum\limits_{i\neq \ell}\int_{\mathbb{R}^n\setminus\Omega_\varepsilon}w_i^pw_\ell\,dx
  \leq C\sum\limits_{i\neq \ell}\int_{\mathbb{R}^n\setminus\Omega_\varepsilon}|x-q_i|^{-p(n+2s)}|x-q_\ell|^{-n-2s}\,dx
  \leq C\varepsilon^{pn+(p+1)2s}\label{J2},
\end{align}
for a constant $C>0$ independent of $q$. %is dependent on the potential $V.$
%\begin{align*}
 % \int_{\Omega_\varepsilon} w_i^pw_l&= \sum\limits_{i=1}^k\int_{\Omega_i}w_i^pw_l\\
 %&\leq C\sum\limits_{i=1}^k\int_{\Omega_i}\frac{1}{(1+|x-q_i|)^{p(n+2s)}}\frac{1}{|x-q_l|^{n+2s}}\\
 %&\leq C\int_{\mathbb{R}^n}\frac{1}{(1+|x-q_i|)^{p(n+2s)}}\frac{1}{|q_i-q_l|^{n+2s}}\\
 %&\leq C d^{-n-2s}.
%\end{align*}
%That is
%\begin{align}\label{J2}
 % J_2\leq C d^{-n-2s}.
%\end{align}
Define
\begin{align*}
  J_3&:=\sum\limits_{i,\ell=1}^k \int_{\Omega_\varepsilon}\{(V(\varepsilon x)-\lambda_i)[w_iw_\ell+(\Lambda_i+\Pi_i)(\Lambda_\ell+\Pi_\ell)-(\Lambda_i+\Pi_i)w_\ell-(\Lambda_\ell+\Pi_\ell)w_i]\}\,dx.
\end{align*}
%where the $J_{4j}, j=1,...,3$ are defined in the last equality.
%\begin{align*}
%J_{41}:&=\int_{\Omega_\varepsilon}(V(\varepsilon x)-\lambda_i)w_i^2,\\
 % J_{42}:&= \int_{\Omega_\varepsilon}(V(\varepsilon x)-\lambda_i)(\Lambda_i+\Pi_\varepsilon(x,q_i))^2,\\
  %\mbox{and} \qquad J_{43}:&=\int_{\Omega_\varepsilon}(V(\varepsilon x)-\lambda_i)(\Lambda_i+\Pi_\varepsilon(x,q_i))w_i.
%\end{align*}
Due to the regularity of $V$ we can apply the mean value theorem to obtain
\begin{equation}\label{eq:VminusLambda}
|V(\varepsilon x)-\lambda_i|\leq C (\varepsilon|\nabla V(\xi_i)| |x-q_i|+\varepsilon^2|x-q_i|^2).
\end{equation}
Hence, using \eqref{behavwi},
\begin{align*}
&\int_{\Omega_\varepsilon} |V(\varepsilon x)-\lambda_i|w_i^2\, dx \leq C \bigg(\varepsilon |\nabla V(\xi_i)|\int_{\Omega_\varepsilon}  \frac{|x-q_i|}{(1+|x-q_i|)^{2(n+2s)}}\,dx+\varepsilon^2 \int_{\Omega_\varepsilon}  \frac{|x-q_i|^2}{(1+|x-q_i|)^{2(n+2s)}}\,dx\bigg)\\
&\qquad \leq C\bigg(\varepsilon |\nabla V(\xi_i)|+\varepsilon^2\bigg).
\end{align*}
Considering $\Omega_i$ defined in \eqref{Omegai},
\begin{align*}
\int_{\Omega_\varepsilon}|V(\varepsilon x)-\lambda_i|w_iw_\ell\leq \sum\limits_{i=1}^k \int_{\Omega_i}|V(\varepsilon x)-\lambda_i|w_i^2\leq C\sum\limits_{i=1}^k \bigg(\varepsilon |\nabla V(\xi_i)|+\varepsilon^2\bigg).
\end{align*}
Hence, using Proposition \ref{prop:boundLambda}, the boundedness of $V$, \eqref{behavwi}, and the fact that $0\leq \Pi_i\leq w_i$ (see Lemma \ref{expressionui}), we conclude that 
\begin{equation}\label{J3}
J_3\leq C\bigg(\varepsilon \sum_{i=1}^k|\nabla V(\xi_i)|+\varepsilon^2\bigg).
\end{equation}
Let us estimate
$$J_4:=\int_{\Omega_\varepsilon}(W_q^{p+1}-U_q^{p+1})\,dx.$$
We expand $W_q^{p+1}$ as
\[W_q^{p+1}(x)=U_q^{p+1}(x)+(p+1)W_q^{p}(x)(W_q(x)-U_q(x))+C\Upsilon_q^{p-1}(x)(W_q(x)-U_q(x))^2\]
where $0\leq U_q(x)\leq\Upsilon_q(x)\leq W_q(x),$ and $C$ is a positive constant depending only on $p.$ Then from \eqref{expressionUq},
\begin{align*}
  J_{4} &\leq C\bigg(\int_{\Omega_\varepsilon}W_q^p\sum\limits_{\ell=1}^k(\Lambda_\ell+\Pi_\ell)
  +\int_{\Omega_\varepsilon}W_q^{p-1}\bigg(\sum\limits_{\ell=1}^k(\Lambda_\ell+\Pi_\ell)\bigg)^{2}\bigg)\\
  &\leq C\bigg(\int_{\Omega_\varepsilon}W_q^p\sum\limits_{\ell=1}^k(\Lambda_\ell+\Pi_\ell)
  +\int_{\Omega_\varepsilon}W_q^{p-1}\bigg(\sum\limits_{\ell=1}^k\Lambda_\ell\bigg)^2+\int_{\Omega_\varepsilon}W_q^{p-1}\bigg(\sum\limits_{\ell=1}^k\Pi_\ell\bigg)^{2}\bigg)\\
    &\leq C\bigg(\int_{\Omega_\varepsilon}W_q^p\sum\limits_{\ell=1}^k(\Lambda_\ell+\Pi_\ell)
  +\int_{\Omega_\varepsilon}W_q^{p-1}\bigg(\sum\limits_{\ell=1}^k\Lambda_\ell\bigg)^2+\int_{\Omega_\varepsilon}W_q^{p}\sum\limits_{\ell=1}^k\Pi_\ell\bigg)\\  
    &\leq C\bigg(\sum_{i,\ell=1}^k\int_{\Omega_\varepsilon}w_i^p(\Lambda_\ell+\Pi_\ell)\,dx+\sum_{i,\ell=1}^k\int_{\Omega_\varepsilon}w_i^{p-1}\Lambda_\ell^2\,dx\bigg),
\end{align*}
where we have used the fact that $\Pi_\ell\leq w_\ell$ (see Lemma \ref{expressionui}) to affirm that 
$\sum_{\ell=1}^k\Pi_\ell\leq W_q$. Using \eqref{J1}, the boundedness of $w_i$ and Proposition \ref{prop:boundLambda} we conclude that 
\begin{equation}\label{J4}
J_4\leq C\bigg(\varepsilon^{n+2s}+\frac{C|\Omega_\varepsilon|}{d^{2p(n+2s)}}\bigg)\leq C\varepsilon^{n+2s}.
\end{equation}
Finally, exploiting $\eqref{behavwi}$ and the fact that $d\geq\frac{\delta_*}{\varepsilon},$ we have
\begin{equation}\begin{split}\label{J5}
  J_5&:=\int_{\mathbb{R}^n\setminus\Omega_\varepsilon}\bigg(\bigg(\frac{1}{p+1}\sum\limits_{i=1}^kw_i\bigg)^{p+1}-\frac{1}{2}\sum\limits_{i=1}^kw_i^{p+1}\bigg)\,dx\leq \frac{1}{2}\int_{\mathbb{R}^n\setminus\Omega_\varepsilon}\bigg(\bigg(\sum\limits_{i=1}^kw_i\bigg)^{p+1}-\sum\limits_{i=1}^kw_i^{p+1}\bigg)\,dx\\
  & \leq \frac{k^p-1}{2}\sum\limits_{i=1}^k\int_{\mathbb{R}^n\setminus\Omega_\varepsilon}w_i^{p+1}\leq C \varepsilon^{pn+(p+1)2s}.
\end{split}\end{equation}
Putting together \eqref{J1}-\eqref{J5} we conclude the result.
\end{proof}

Given $q\in\Xi_\eta$ and $U_q$, let us denote by $\Phi(q)\in X$ the unique solution to \eqref{projectedp} provided by Proposition \ref{Phi}. Then
\begin{equation}\label{defuq}
u_q:=U_q+\Phi(q),
\end{equation}
satisfies the equation
\begin{equation}\label{uq}
  (-\Delta)^{s} u_q + V(\varepsilon x)u_q -u_q^{p}=\sum\limits_{i=1}^k\sum_{j=1}^n c_{ij}Z_{ij}  \,\, \mbox{in }\; \Omega_\varepsilon,
\end{equation}
with $Z_{ij}$ specified in \eqref{deZij}. Let us define the function $I_\varepsilon: \Xi_\eta\rightarrow\mathbb{R}$ as
\begin{equation}\label{defIq}
I_\varepsilon(q):=J_\varepsilon(u)=J_\varepsilon(U_q+\Phi(q)),
\end{equation}
with $J_\varepsilon$ given in \eqref{Jepsilon}.

\begin{lemma} \label{equivalence} If $\varepsilon>0$ is small enough, the coefficients $c_{ij}$ in \eqref{uq} are equal to zero for all $i\in\{1,\ldots, j\}$, $j\in\{1,\ldots, k\},$ if and only if
\[\frac{\partial I_\varepsilon(q)}{\partial q}:=\left(\frac{\partial I_\varepsilon(q)}{\partial q_{11}},\ldots, \frac{\partial I_\varepsilon(q)}{\partial q_{1n}},\ldots,\frac{\partial I_\varepsilon(q)}{\partial q_{k1}},\ldots, \frac{\partial I_\varepsilon(q)}{\partial q_{kn}}\right)=0.\]
\end{lemma}

\begin{proof}
From Lemma \ref{limit1} and Corollary \ref{cor:derw_qij}, we have that
\[\frac{\partial U_q}{\partial q_{\ell m}}=\frac{\partial W_q}{\partial q_{\ell m}}+O(\varepsilon^{\nu_1})=-Z_{\ell m}+O(\varepsilon), \qquad \ell=1,\ldots, k,\; m=1,\ldots, n,\]
and, from Lemmata \ref{estimateE}, \ref{estimateE1} and Proposition \ref{Phi1}, 
\[\frac{\partial \Phi(q)}{\partial q_{\ell m}}=O\left(\varepsilon^{\min\{1,(p-1)(n+2s)\}} +\eta^{\mu-n-2s}\right),\]
where $\eta :=\min\{|q_i-q_\ell|:i\neq \ell\}$. 
Therefore,
\[\frac{\partial u_q}{\partial q_{\ell m}}=-Z_{\ell m}+O\left(\varepsilon^{\min\{1,(p-1)(n+2s)\}} +\eta^{\mu-n-2s}\right),\]
and then, using Lemma \ref{Zij}, there exists $C>0$ such that
\[\left|\frac{\partial u_q}{\partial q_{\ell m}}\right|\leq C.\]
Proceeding like in \cite[Lemma 7.16]{DdPDV2015} we deduce
\begin{align}\label{eq:derqm}
  \frac{\partial I_\varepsilon(q)}{q_{\ell m}}
   &=\sum\limits_{i=1}^k\sum_{j=1}^n c_{ij}\int_{\Omega_\varepsilon}Z_{ij}\left(-Z_{\ell m}+O\left(\varepsilon^{\min\{1,(p-1)(n+2s)\}} +\eta^{\mu-n-2s}\right)\right)\,dx =-\sum\limits_{i=1}^k\sum_{j=1}^nc_{ij}M_{ij}^{\ell m}
   \end{align}
with 
$$M_{ij}^{\ell m}:=\alpha_i\delta_{i\ell}\delta_{jm}+O\left(\varepsilon^{\min\{1,(p-1)(n+2s)\}} +\eta^{\mu-n-2s}\right),$$
as a consequence of Corollary \ref{orthogonal1}. Notice that, by Lemma \ref{orthogonal}, $\alpha_i>0$. In order to write the equation in matricial form, notice that, given any number $\gamma\in\{1,\ldots,kn\}$ this can be univoquely written as
$$\gamma=(i-1)n+j\quad \mbox{ for certain (unique)}\quad i\in \{1,\ldots, k\},\, j\in \{1,\ldots, n\}.$$
Hence, using \eqref{eq:derqm} we can write
$$\frac{\partial I_\varepsilon(q)}{\partial q}= \mathbf{M} \mathbf{c}^t,$$
where $\mathbf{M}$ is a $kn\times kn$ matrix and $\mathbf{c}$ is a $kn$ dimensional vector whose entries are
$$\mathbf{M}_{r\gamma}:=M_{ij}^{\ell m}, \quad \mathbf{c}_r=c_{\ell m},$$
where
$$r=(\ell-1)n+m,\quad \gamma =(i-1)n+j,\qquad \ell,i \in \{1,\ldots, k\},\;\;m,j\in \{1,\ldots, n\}.$$
Using the fact that $\alpha_i>0$ (see Lemma \ref{orthogonal}) for every $i\in \{1,\ldots, k\}$ the invertibility of $\mathbf{M}$ follows for $\varepsilon$ small enough and $\eta$ sufficiently large, and the result holds. 
\end{proof}

%\begin{lemma} \label{expansion} Given $q\in \Xi_\eta$, the following expansions hold:
%\[I_\varepsilon(q)=J_\varepsilon(U_q)+O(\tau^2),\qquad \frac{\partial I_\varepsilon(q)}{\partial q}=\frac{\partial J_\varepsilon(U_q)}{\partial q}+O(\tau^{\min\{p,2\}}+\tau\tilde{\tau}),\]
%where $\tau:=\varepsilon +\eta^{\mu-n-2s}$ and $\tilde{\tau}:=\varepsilon^{\min\{1,(p-1)(n+2s)\}}+\eta^{\mu-n-2s}$, with  $\eta :=\min\{|q_i-q_\ell|:i\neq \ell\}$.
%\end{lemma}
\begin{lemma} \label{expansion} The following expansion hold:
\[I_\varepsilon(q)=J_\varepsilon(U_q)+O(\tau^2),\]
where $\tau:=\varepsilon +\eta^{\mu-n-2s}$.
\end{lemma}

\begin{proof}Using definitions \eqref{defuq} and \eqref{defIq},
\begin{align*}
 I_\varepsilon(q)& =J_\varepsilon(U_q)+\int_{\Omega_\varepsilon}\Phi(q)[(-\Delta)^sU_q+V(\varepsilon x) U_q- U_q^{p}]\,dx+\frac{1}{2}\int_{\Omega_\varepsilon}(-\Delta)^s\Phi(q)\Phi(q)+V(\varepsilon x) \Phi^2(q)\,dx\\
   &\quad-\frac{1}{p+1}\int_{\Omega_\varepsilon} [(U_q+\Phi(q))^{p+1}-U_q^{p+1}-(p+1)U_q^p \Phi(q)]\,dx\\
   &=J_\varepsilon(U_q)+\int_{\Omega_\varepsilon}\left((-\Delta)^su_q+V(\varepsilon x) u_q- u_q^{p}\right)\Phi(q)\,dx\\
   &\quad-\int_{\Omega_\varepsilon}\left(-\Delta)^s(u_q-U_q)+V(\varepsilon x) (u_q-U_q)\right)\Phi(q)\,dx+\frac{1}{2}\int_{\Omega_\varepsilon}(-\Delta)^s\Phi(q)\Phi(q)+V(\varepsilon x) \Phi^2(q)\,dx\\
   &\quad +\int_{\Omega_\varepsilon}\left(u_q^{p}-U_q^{p}\right)\Phi(q)-\frac{1}{p+1}\int_{\Omega_\varepsilon} [(U_q+\Phi(q))^{p+1}-U_q^{p+1}-(p+1)U_q^p \Phi(q)]\,dx. 
\end{align*}
Since $u_q$ solves \eqref{uq} and 
\begin{equation}\label{eq:Phiort}
\int_{\Omega_\varepsilon}\Phi(q)Z_{ij}\,dx=0\qquad \mbox{ for every }i\in\{1,\ldots, k\},\;j\in \{1,\ldots,n\},
\end{equation}
we can reduce it to 
\begin{align*}
 I_\varepsilon(q)& =J_\varepsilon(U_q)-\frac{1}{2}\int_{\Omega_\varepsilon}(-\Delta)^s\Phi(q)\Phi(q)+V(\varepsilon x) \Phi^2(q)\,dx+\int_{\Omega_\varepsilon}\left(u_q^{p}-U_q^{p}\right)\Phi(q)\,dx\\
 &\quad -\frac{1}{p+1}\int_{\Omega_\varepsilon}[(U_q+\Phi(q))^{p+1}-U_q^{p+1}-(p+1)U_q^p \Phi(q)]\,dx.
\end{align*}
Applying Lemma \ref{ineq} and Lemma \ref{limit},
\begin{align*}
\left|\int_{\Omega_\varepsilon}\left(u_q^{p}-U_q^{p}\right)\Phi(q)\,dx\right|&\leq C\int_{\Omega_\varepsilon}|U_q|^{p-1}\Phi^2(q)\,dx\leq C\int_{\Omega_\varepsilon}|W_q+\varepsilon^{n+2s}|^{p-1}\Phi^2(q)\,dx\\
&\leq C\|\Phi(q)\|^2_*\int_{\Omega_\varepsilon}|W_q+\varepsilon^{n+2s}|^{p-1}\rho_q^2\,dx \leq C \tau^2,
\end{align*}
where $C$ is independent of $q$. Likewise, by Lemma \ref{ineq},
\begin{align*}
&\left|\int_{\Omega_\varepsilon}\left(u_q^{p+1}-U_q^{p+1}-(p+1)U_q^{p}\Phi(q)\right)dx\right|\leq C\int_{\Omega_\varepsilon}|U_q|^{p-1}\Phi^2(q)\,dx\leq C \tau^2.
\end{align*}
Since $\Phi$ satisfies equation \eqref{dirichletProj} and \eqref{eq:Phiort}, using Lemmata \ref{estimateN}, \ref{estimateE} and Proposition \ref{Phi}, we have that
\begin{align*}
&\left|\int_{\Omega_\varepsilon}(-\Delta)^s\Phi(q)\Phi(q)+V(\varepsilon x) \Phi^2(q)\,dx\right|\leq C(\|E(\Phi)\|_*+\|N(\Phi)\|_*+\|\Phi\|_*)\|\Phi\|_*\leq C\tau^2,
\end{align*}
where $C$ is independent of $q$. Therefore
\[I_\varepsilon(q)=J_\varepsilon(U_q)+O(\tau^2).\]
\end{proof}

\begin{lemma}\label{criticalV1bis}
Let $q\in \Xi_\eta$ and denote $\xi_i=\varepsilon q_i$. Assume that $u_q$ is a solution of \eqref{uq}. Then, for every $i=\{1,\ldots, k\}$ and $j\in\{1,\ldots, n\}$,
$$c_{ij}=\varepsilon \gamma_i \frac{\partial V(\xi_i)}{\partial x_j}+O\big(\varepsilon^{\frac 12} \tau+\eta^{-\min\{1,p-1\}(n+2s)}\tau+\varepsilon^{n+2s}+\varepsilon\eta^{-\nu_1+1}+\eta^{\mu-n-2s}\big),$$
provided $\eta\to \infty$ as $\varepsilon\to 0$, with $\tau:=\varepsilon+\eta^{\mu-n-2s}$ and $\gamma_i$ a positive constant independent of $j$.
\end{lemma}

\begin{proof}
If $u_q=U_q+\Phi(q)$ is a solution of \eqref{uq}, then $\Phi(q)$ solves \eqref{projectedp}, and hence, by Lemma \ref{cij},
\begin{equation}\begin{split}\label{value_cij}
c_{ij}&=-\frac{1}{\alpha_i}\int_{\Omega_\varepsilon}(E(\Phi)+N(\Phi))Z_{ij}\,dx+O(\varepsilon^{\frac{1}{2}}+\eta^{-\min\{1,p-1\}(n+2s)})(\|\Phi\|_{*}+\|E(\Phi)\|_{*}+\|N(\Phi)\|_{*})\\
&=-\frac{1}{\alpha_i}\int_{\Omega_\varepsilon}E(\Phi)Z_{ij}\,dx+O(\varepsilon^{\frac{1}{2}}\tau +\eta^{-\min\{1,p-1\}(n+2s)}\tau+\|N(\Phi)\|_*),
\end{split}\end{equation}
where in the last inequality we used Theorem \ref{Phi}, Lemma \ref{estimateN} and Lemma \ref{estimateE}.
Likewise, using \eqref{eq:E1} and \eqref{eq:E3},
\begin{equation*}\begin{split}
\int_{\Omega_\varepsilon}&E(\Phi)Z_{ij}\,dx = \sum_{\ell=1}^k\int_{\Omega_\varepsilon}(V(\varepsilon x)-\lambda_\ell)\bar{u}_\ell Z_{ij}\,dx+O(\varepsilon^{n+2s}+\eta^{\mu-n-2s}).
\end{split}\end{equation*}
 By Lemma \ref{limit},
\begin{align*}
\sum_{\ell=1}^k\int_{\Omega_\varepsilon}&(V(\varepsilon x)-\lambda_\ell)\bar{u}_\ell Z_{ij}\,dx =\int_{\Omega_\varepsilon}(V(\varepsilon x)-\lambda_i)U_q Z_{ij}\,dx+\int_{\Omega_\varepsilon}\sum_{\ell\neq i}^k(\lambda_i-\lambda_\ell)\bar{u}_\ell Z_{ij}\,dx\\
&=\int_{\Omega_\varepsilon}(V(\varepsilon x)-\lambda_i)W_q Z_{ij}\,dx+\int_{\Omega_\varepsilon}\sum_{\ell\neq i}^k(\lambda_i-\lambda_\ell)w_\ell Z_{ij}\,dx+O(\varepsilon^{n+2s})\\
&=\varepsilon\int_{\Omega_\varepsilon}\nabla V(\xi_i)\cdot(x-q_i)W_q Z_{ij}\,dx+\int_{\Omega_\varepsilon}\sum_{\ell\neq i}^k(\lambda_i-\lambda_\ell)w_\ell Z_{ij}\,dx+O(\varepsilon^{n+2s}).
\end{align*}
Consider $\Omega_i$, defined in \eqref{Omegai}. Since
\begin{equation}\label{distqell}
|x-q_\ell|\geq \frac{|q_i-q_\ell|}{2}\geq \frac{\eta}{2} \quad \mbox{ for every }x\in \Omega_i,\;\; i\neq \ell,
\end{equation}
applying Lemma \ref{Zij}, \eqref{behavwi} and Remark \ref{unifV}, then 
\begin{equation*}
\bigg|\int_{\Omega_\ell}(\lambda_i-\lambda_\ell)w_\ell Z_{ij}\,dx\bigg|\leq C\varepsilon \eta^{-\nu_1+1},\qquad \bigg|\int_{\Omega_i}(\lambda_i-\lambda_\ell)w_\ell Z_{ij}\,dx\bigg|\leq C\varepsilon \eta^{-n-2s+1},
\end{equation*}
for a positive constant $C$ independent of $q$. Therefore, 
\begin{align}\label{INT1}
\sum_{\ell=1}^k\int_{\Omega_\varepsilon}&(V(\varepsilon x)-\lambda_\ell)\bar{u}_\ell Z_{ij}\,dx=\varepsilon\frac{\partial V(\xi_i)}{\partial x_j}\int_{\Omega_\varepsilon}(x_j-q_{ij})W_q Z_{ij}\,dx+O(\varepsilon^{n+2s}+\varepsilon \eta^{-\nu_1+1}+\varepsilon \eta^{-n-2s+1}).
\end{align}
%Plugging \eqref{INT3}, \eqref{INT4}, \eqref{INT1} and \eqref{INT4} into \eqref{eq:expansionGrad}, and using the definition of $\tau$, we deduce
%\begin{equation}\label{eq:partialjV}
%\varepsilon\frac{\partial V(\xi_i)}{\partial x_j}\int_{\Omega_\varepsilon}(x_j-q_{ij})W_q Z_{ij}\,dx+O(\tau^2+\tau^p+\eta^{-n-2s}+\varepsilon \eta^{-\nu_1+1}+\varepsilon \eta^{-n-2s+1})=0.
%\end{equation}
Notice that,  using \eqref{eq:radialZij},
\begin{equation}\begin{split}\label{compwiZij}
\int_{\Omega_\varepsilon}(x_j-q_{ij})W_q Z_{ij}\,dx&=\int_{\Omega_\varepsilon}(x_j-q_{ij})w_i Z_{ij}\,dx+\sum_{\ell\neq i}\int_{\Omega_\varepsilon}(x_j-q_{ij})w_\ell Z_{ij}\,dx\\
&=\int_{\R^n}\frac{y_j^2}{|y|} w(|y|)w'(|y|)\,dy+O(\varepsilon^{\nu_1+2s-1}+\eta^{-n-2s}+\eta^{-\nu_1+1})\\
&=c_0+O(\varepsilon^{\nu_1+2s-1}+\eta^{-n-2s}+\eta^{-\nu_1+1}),
\end{split}\end{equation}
with $c_0<0$ independent of $j$.

Substituting in \eqref{value_cij} we conclude
$$c_{ij}=\varepsilon \gamma_i \frac{\partial V(\xi_i)}{\partial x_j}+O\big(\varepsilon^{\frac 12} \tau+\eta^{-\min\{1,p-1\}(n+2s)}\tau+\varepsilon^{n+2s}+\varepsilon\eta^{-\nu_1+1}+\eta^{\mu-n-2s}\big),$$
with $\gamma_i=-\frac{c_0}{\alpha_i}>0$ independent of $j$.
\end{proof}

\begin{lemma}\label{criticalV} Let $\theta:=\frac{p+1}{p-1}-\frac{n}{2s}$ and $q\in\Xi_\eta$. Denoting $\xi=\varepsilon q$, the following expansion holds:
\begin{align*}
  I_\varepsilon(q) &=c_*\sum\limits_{i=1}^kV^\theta(\xi_i)-\sum\limits_{i=1}^k\sum\limits_{i\neq \ell}\frac{c_{i\ell}^*}{|q_i-q_\ell|^{n+2s}}+O(\bar{\tau}),
%  \\ \nabla_q I_\varepsilon(q) &=\varepsilon\nabla_\xi\bigg[c_*\sum\limits_{i=1}^kV^\theta(\xi_i)\bigg]+O(\hat{\tau})
\end{align*}
where 
\begin{align*}
c_*&:=\frac{1}{2}\int_{\mathbb{R}^n}(w(-\Delta)^s w+w^2)\,dx-\frac{1}{p+1}\int_{\mathbb{R}^n}w^{p+1}\,dx,\\
c_{i\ell}^*&=\frac{1}{2}c_0 V(\xi_i)^{\frac{p}{p-1}-\frac{n}{2s}}V(\xi_\ell)^{\frac{1}{p-1}-\frac{n+2s}{2s}}\int_{\mathbb{R}^n}w^p\,dx,
\end{align*}
and
\begin{align*}
\bar{\tau}&:=\varepsilon\sum\limits_{i=1}^k |\nabla V(\xi_i)|+\varepsilon^2+\eta^{-2(n+2s-\mu)}+\eta^{-\min\{2,p\}(n+2s)}.\\
%\hat{\tau}&:=\varepsilon\sum\limits_{i=1}^k |\nabla V(\xi_i)|+\varepsilon^{\min\{1,(p-1)(n+2s)\}}\eta^{-(n+2s-\mu)}+\varepsilon^{\min\{p,2\}}+\eta^{-(n+2s)}.
\end{align*}
\end{lemma}

\begin{proof}
Applying Lemma \ref{expansion} and Proposition \ref{propenergy}, we have that
\[I_\varepsilon(q)= J(W_q)%+\frac{1}{2}\sum\limits_{i=1}^k\int_{\Omega_\varepsilon} w_i^{p}\Pi_\varepsilon(x,q_i)
+O(\bar{\tau}),\qquad \bar{\tau}:=\varepsilon\sum\limits_{i=1}^k |\nabla V(\xi_i)|+(\varepsilon+\eta^{\mu-n-2s})^2,\]
where (recall \eqref{JWq}),
\begin{align*}
               J(W_q)=\sum\limits_{i=1}^kJ_i(w_i)+ \frac{1}{2}\sum\limits_{i\neq \ell}\int_{\mathbb{R}^n} w_i^pw_\ell\,dx -\frac{1}{p+1}\int_{\mathbb{R}^n}\Bigl[\Big(\sum\limits_{i=1}^kw_i\Big)^{p+1}-\sum\limits_{i=1}^kw_i^{p+1}\Big]\,dx.
             \end{align*}
Furthermore, it can be easily seen (see definition \eqref{Ji}) that 
\begin{align}\label{eq_Jiwi}
  J_i(w_i)   &=\lambda_i^{\frac{p+1}{p-1}-\frac{n}{2s}}\left[\frac{1}{2}\int_{\mathbb{R}^n}(w(-\Delta)^s w+ w^2)\,dx -\frac{1}{p+1}\int_{\mathbb{R}^n}w^{p+1}\,dx\right]=V^{\theta}(\xi_i)c_*.
\end{align}
Using \eqref{behavw} and \eqref{behavwi}, after a change of variable we can prove that 
\begin{equation}\label{eq:wiwl}
\int_{\mathbb{R}^n}w_i^pw_\ell\,dx=\lambda_i^{\frac{p}{p-1}-\frac{n}{2s}}\lambda_\ell^{\frac{1}{p-1}-\frac{n+2s}{2s}}
\int_{\mathbb{R}^n}w^p\frac{c_0}{|q_i-q_\ell|^{n+2s}}(1+o_\varepsilon(1)),
\end{equation}
provided that $\eta\to\infty$ as $\varepsilon \to 0$. Here $o_\varepsilon(1)$ stands for a quantity that converges to 0 whenever $\varepsilon \to 0$.
%then
%\[c_{i\ell}=c_0 \lambda_i^{\frac{p}{p-1}-\frac{n}{2s}}\lambda_\ell^{\frac{1}{p-1}-\frac{n+2s}{2s}}
%=c_0 V(\xi_i)^{\frac{p}{p-1}-\frac{n}{2s}}V(\xi_\ell)^{\frac{1}{p-1}-\frac{n+2s}{2s}}.\]
%%See \cite[Lemma A.2]{DR} and \cite[Lemma 2.2]{KW2000}.
Splitting the integral into the subdomains $\Omega_i$ (see \eqref{Omegai}) and using \eqref{eq:wiwl},
\begin{align*}
 \int_{\mathbb{R}^n}&\Bigl[\Big(\sum\limits_{i=1}^kw_i\Big)^{p+1}-\sum\limits_{i=1}^kw_i^{p+1}\Bigl]\,dx =\int_{\Omega_\varepsilon}\Bigl[\Big(\sum\limits_{i=1}^kw_i\Big)^{p+1}-\sum\limits_{i=1}^kw_i^{p+1}\Bigl]\,dx+O\big(\varepsilon^{pn+(p+1)2s}\big)\\
  & =\sum\limits_{i=1}^k\int_{\Omega_i}\Bigl((p+1)w_i^p\sum\limits_{ \ell\neq i}w_\ell+O\Big(w_i^{\min\{p-1,1\}}\Big(\sum\limits_{i\neq \ell}w_\ell\Big)^2\Big)\Bigl)\,dx+O\big(\varepsilon^{pn+(p+1)2s}\big)\\
  &=\sum\limits_{i=1}^k\sum\limits_{\ell\neq i}(p+1)\int_{\R^n}w_i^pw_\ell\,dx+O\big(\varepsilon^{pn+(p+1)2s}+\eta^{-\min\{2,p\}(n+2s)}\big)\\
  &=\sum\limits_{i=1}^k\sum\limits_{\ell\neq i}(p+1)\int_{\mathbb{R}^n}w^p\frac{c_{i\ell}}{|q_i-q_\ell|^{n+2s}}(1+o_\varepsilon(1))\,dx+O\big(\varepsilon^{pn+(p+1)2s}+\eta^{-\min\{2,p\}(n+2s)}\big), 
\end{align*}
with
\[c_{i\ell}:=c_0 V(\xi_i)^{\frac{p}{p-1}-\frac{n}{2s}}V(\xi_\ell)^{\frac{1}{p-1}-\frac{n+2s}{2s}}.\]
Hence
\[J(W_q)=c_*\sum\limits_{i=1}^kV^\theta(\xi_i)-\sum\limits_{i=1}^k\sum\limits_{\ell\neq i}\frac{c_{i\ell}^*}{|q_i-q_\ell|^{n+2s}}(1+o_\varepsilon(1))+O\big(\varepsilon^{pn+(p+1)2s}+\eta^{-\min\{2,p\}(n+2s)}\big),\]
and
\[I_\varepsilon(q)=c_*\sum\limits_{i=1}^kV^\theta(\xi_i)-\sum\limits_{i=1}^k\sum\limits_{\ell\neq i}\frac{c_{i\ell}^*}{|q_i-q_\ell|^{n+2s}}(1+o_\varepsilon(1))+O(\bar{\tau}).\]
\end{proof}

\section{Proofs of the Theorems }

In this section, we present the proofs of Theorems \ref{result1}-\ref{result4}. We begin with the existence results.
\medskip

\textbf{Proof of Theorem \ref{result1}.} Fix $\eta=c_\star \varepsilon^{-1}$, $c_\star>0$, in the definition of the configuration space $\Xi_\eta$ (see \eqref{Xi}). Then, given $q\in\Xi_\eta$, the function
$u_q=U_q+\Phi(q),$
with $\Phi(q)$ provided by Theorem \ref{Phi}, solves \eqref{projectedp}. For this choice of $\eta$, from Lemma \ref{criticalV1bis} we get
$$c_{ij}=\varepsilon\bigg[\gamma_i\frac{\partial V(\xi_i)}{\partial x_j}+o_\varepsilon(1)\bigg],\qquad \xi_i:=\varepsilon q_i,\quad i\in\{1,\ldots, k\},$$
where $o_\varepsilon(1)$ stands for a quantity that converges to 0 as $\varepsilon\to 0$, uniformly in $q\in\Xi_\eta$.

By virtue of the hypotheses on $V$, there exists $\hat{\xi}=(\hat{\xi}_1,\ldots,\hat{\xi}_k)\in \varepsilon \Xi_\eta$ such that 
$$\nabla V(\hat{\xi}_i)=0\quad \mbox{ and }\quad  \mbox{det}(D^2V(\hat{\xi}_i))\neq 0,$$
for every $i=\{1,\ldots,k\}$. Hence, by the implicit function theorem, there exists a point $\xi^\varepsilon_i\in \varepsilon\Xi_\eta$ such that $|\xi_i^\varepsilon-\hat{\xi_i}|\to 0$ when $\varepsilon\to 0$ and
$$\gamma_i\frac{\partial V(\xi_i^\varepsilon)}{\partial x_j}+o_\varepsilon(1)=0,$$
for every $i\in\{1,\ldots,k\}$, $j\in\{1,\ldots, n\}$. Thus, 
$$c_{ij}^\varepsilon:=\varepsilon\bigg[\gamma_i\frac{\partial V(\xi_i^\varepsilon)}{\partial x_j}+o_\varepsilon(1)\bigg]=0,$$
and, denoting $q^\varepsilon:=\frac{\xi^\varepsilon}{\varepsilon}$,
$$u_{q^\varepsilon}:=U_{q^\varepsilon}+\Phi(q^\varepsilon),$$
is a solution of \eqref{dirichlet1}, with $\|\Phi(q^\varepsilon)\|_*\to 0$ as $\varepsilon\to 0$. 
$\hfill\square$\\

\textbf{Proof of Theorem \ref{result3}.} Following the ideas in \cite[Proposition 4.2]{KW2000}, we consider the following configuration space
\[A=\{\xi=(\xi_1,\cdots,\xi_k):\xi_i\in K,\;\;\min\limits_{i\neq \ell}|\xi_i-\xi_\ell|>\varepsilon^{1-\frac{\alpha}{n+2s}}\},\qquad \alpha\in (0,1),\]
where $K$ is given in the theorem. Let us fix 
\begin{equation}\label{choiceEta}
\eta=\tfrac{1}{2}\varepsilon^{-\frac{\alpha}{n+2s}}\quad \mbox{ in such a way that }\quad A\subset \Xi_\eta,
\end{equation}
and consider the functional $I_\varepsilon(q)$ given in \eqref{defIq}, for $q:=\frac{\xi}{\varepsilon}$, $\xi\in A$.
Since it is continuous, $I_\varepsilon$ admits a maximizer $q^\varepsilon=\frac{\xi^\varepsilon}{\varepsilon}$, with $\xi^\varepsilon\in\bar{A}.$ Let us see that actually $\xi^\varepsilon\in A$.

Choose a point $\hat{\xi}\in K$ such that $V(\hat{\xi})=\max\limits_{K}V$ (this point exists by the conditions on $K$ and $V$) and define
 \[\xi_i^0:=\hat{\xi}+\varepsilon^{1-\frac{\beta}{n+2s}}X_i,\quad i=1,\cdots,k,\]
where every $X_i$ is a vertex of a $k$-polygon centered at $0$ with $|X_i-X_\ell| = 1$ for $i\neq \ell$, and 
\begin{equation}\label{rangebeta}
\beta\in (0,1)\quad \mbox{ such that }\quad \alpha<\beta<\alpha\min\bigg\{p,2,2\frac{n+2s-\mu}{n+2s}\bigg\}.
\end{equation}
Notice that this range is admissible due to the definiton of $\mu$ (see \eqref{mu}) and the fact that $p>1$, which make
$$\min\bigg\{p,2,2\frac{n+2s-\mu}{n+2s}\bigg\}>1.$$
Since $K$ is open, taking $\varepsilon$ small enough we can assume $\xi_i^0\in K.$ Furthermore, $\xi^0=(\xi_1^0,\cdots,\xi_k^0)\in A$ since $\alpha<\beta$. Denoting $q^0:=\frac{\xi^0}{\varepsilon}$, by Lemma \ref{criticalV} we have
 \begin{equation}\begin{split}\label{lowbdd}
  I_\varepsilon(q^\varepsilon) &=\max\limits_{\bar{A}} I_\varepsilon(q)\geq I_\varepsilon(q^0)\geq c_*k \sup\limits_{x\in K}V^\theta(x)-c_1\varepsilon^{\beta}+O\big(\varepsilon +\eta^{-\min\{p,2\}(n+2s)}+\eta^{-2(n+2s-\mu)}\big)\\
  &\geq c_*k \sup\limits_{x\in K}V^\theta(x)-c_2\varepsilon^{\beta},
\end{split}\end{equation}
due to \eqref{choiceEta} and \eqref{rangebeta}, where $c_1,c_2$ are positive constants.

Suppose $\xi^\varepsilon\in \partial A.$ Then either there is an index $i$ such that $\xi_i^\varepsilon\in \partial K,$ or there exist indices $i\neq \ell$ such that
\[|\xi_i^\varepsilon-\xi_\ell^\varepsilon|=\min\limits_{i\neq \ell}|\xi_i-\xi_\ell|=\varepsilon^{1-\frac{\alpha}{(n+2s)}}.\]
In the first case, from Lemma \ref{criticalV} and \eqref{rangebeta}, we see that
\begin{align*}
  I_\varepsilon(q^\varepsilon)&\leq c_* V^\theta(\xi_i^\varepsilon)+c_*\sum\limits_{\ell\neq i}V^\theta(\xi_\ell^\varepsilon)+C\varepsilon^\beta\leq c_* k \max\limits_{x\in K}V^\theta(x)+c_*\left(\max\limits_{x\in \partial K}V^\theta(x)-\max\limits_{x\in K}V^\theta(x)\right)+C\varepsilon^\beta,
\end{align*}
which contradicts \eqref{lowbdd} since by hypothesis
$$\max\limits_{x\in \partial K}V^\theta(x)-\max\limits_{x\in K}V^\theta(x)\leq -\rho_0<0,$$
with $\rho_0$ independent of $\varepsilon$. In the second case, applying Lemma \ref{criticalV} again,
\[I_\varepsilon(q^\varepsilon)\leq  c_* k \max\limits_{K}V^\theta(x)-c_3\varepsilon^{\alpha}+O(\varepsilon^\beta),\]
for some positive $c_3,$ which is also a contradiction with \eqref{lowbdd} for $\varepsilon$ sufficiently small since $\alpha<\beta$.
Therefore, necesarilly $\xi^\varepsilon\in A$ and hence
\[\nabla I_\varepsilon(q^\varepsilon)=0\]
since the domain $A$ is open. Therefore, Lemma \ref{equivalence} applies and we conclude that
$$u_{q^\varepsilon}:=U_{q^\varepsilon}+\Phi(q^\varepsilon),$$
is a solution of \eqref{dirichlet1}, with $\xi^\varepsilon_i\in K$ for every $i\in\{1,\ldots, k\}$, and $\|\Phi(q^\varepsilon)\|_*\to 0$ as $\varepsilon\to 0$. This proves the first part of the Theorem \ref{result3}.
%we completed the proof.

Finally, if $\hat{\xi}$ is a strict local maximum of $V$, we can take $\varepsilon$ small enough such that
$$V(\hat{\xi})>V(\xi)\qquad \mbox{ for every }\xi\in B_{\rho_\varepsilon}(\hat{\xi})\setminus \{\hat{\xi}\},\qquad \rho_\varepsilon:=2\varepsilon^{1-\frac{\beta}{n+2s}}.$$
Repeating the previous argument with $K=B_{\rho_\varepsilon}(\hat{\xi})$ we find a $k$-spike solution whose peaks $\xi_i^\varepsilon$ satisfy
$$V(\xi_i^\varepsilon)\to V(\hat{\xi})\qquad \mbox{ as }\varepsilon \to 0.$$ $\hfill\square$\\

In order to prove the non existence result, we need the following improved expansion for potential solutions:

\begin{lemma} \label{nonexistence2}
Suppose that $\hat{\xi}$ is a local minimum point of $V$ such that $det(D^2 V(\hat{\xi}))\neq 0$, and assume $u_\varepsilon$ is a solution to  \eqref{dirichlet} of the form \eqref{eq:u_eps} with $\xi_i^\varepsilon\rightarrow\hat{\xi}$ as $\varepsilon\to 0$.
Then
\[-\varepsilon\frac{\partial V(\xi_i^\varepsilon)}{\partial x_j}+c\sum\limits_{\ell\neq i}\frac{1}{|q_\ell^\varepsilon-q_i^\varepsilon|^{n+2s}}\left(\frac{q_\ell^\varepsilon-q_i^\varepsilon}{|q_\ell^\varepsilon-q_i^\varepsilon|}\right)_j+O(\varepsilon^{2})+o(\eta^{-(n+2s)})=0\]
for $i =1,...,k, j =1,...,n,$ where $c>0$ is a positive number.
\end{lemma}

\begin{proof} 
We refine the arguments in the proof of Lemma \ref{criticalV1bis} taking advantage of the hypothesis of $\xi_i^\varepsilon$. 

Indeed, since $u_\varepsilon=U_{q^\varepsilon}+\Phi(q^{\varepsilon})$ is a solution, $c_{ij}=0$ for every $i\in\{1,\ldots, k\}$, $j\in\{1,\ldots, n\}$, and from \eqref{value_cij} we deduce
\begin{equation*}
\int_{\Omega_\varepsilon}E(\Phi)Z_{ij}\,dx+O(\varepsilon^{\frac{1}{2}}\tau +\eta^{-\min\{1,p-1\}(n+2s)}\tau+\tau^2)=0.
\end{equation*}
Actually, using Proposition \ref{Phi1} and applying the ideas in \cite[Lemma 3.2]{ARS2021} (see the estimate of $A_4$), this identity can be improved to
\begin{equation}\label{ec0}
\int_{\Omega_\varepsilon}E(\Phi)Z_{ij}\,dx+O(\varepsilon^{2}) +o(\eta^{-n-2s})=0.
\end{equation}
Likewise, using \eqref{eq:E1},
\begin{equation}\begin{split}\label{ec1}
\int_{\Omega_\varepsilon}&E(\Phi)Z_{ij}\,dx = \sum_{\ell=1}^k\int_{\Omega_\varepsilon}(V(\varepsilon x)-\lambda_\ell)\bar{u}_\ell Z_{ij}\,dx+\int_{\Omega_\varepsilon}\bigg(\bigg(\sum_{\ell=1}^kw_\ell\bigg)^p-\sum_{\ell=1}^kw_\ell^p\bigg)Z_{ij}\,dx+O(\varepsilon^{n+2s}).
\end{split}\end{equation}
By Lemma \ref{limit},
\begin{equation*}\begin{split}
\sum_{\ell=1}^k&\int_{\Omega_\varepsilon}(V(\varepsilon x)-\lambda_\ell)\bar{u}_\ell Z_{ij}\,dx=\int_{\Omega_\varepsilon}(V(\varepsilon x)-\lambda_i)w_i Z_{ij}\,dx+\sum_{\ell\neq i}^k\int_{\Omega_\varepsilon}(V(\varepsilon x)-\lambda_\ell)w_\ell Z_{ij}\,dx+O(\varepsilon^{n+2s})\\
&=\int_{\Omega_\varepsilon}(V(\varepsilon x)-\lambda_i)w_i Z_{ij}\,dx+\sum_{\ell\neq i}^k\int_{\Omega_\varepsilon}(V(\varepsilon x)-V(\hat{\xi}))w_\ell Z_{ij}\,dx+\sum_{\ell\neq i}^k\int_{\Omega_\varepsilon}(V(\hat{\xi})-\lambda_\ell)w_\ell Z_{ij}\,dx+O(\varepsilon^{n+2s}).
\end{split}\end{equation*}
Let $R_0<\eta$ and denote $\hat{q}=\frac{\hat{\xi}}{\varepsilon}$. Then, doing a Taylor expansion and using the fact that $\hat{\xi}$ is a critical point of $V$,
\begin{align*}
&\bigg|\int_{B_{R_0}(\hat{q})}(V(\varepsilon x)-V(\hat{\xi}))w_\ell Z_{ij}\,dx\bigg|\leq C \varepsilon^2\int_{B_{R_0}(\hat{q})}|x-\hat{q}|^2|w_\ell||Z_{ij}|\,dx\leq C\varepsilon^2.
\end{align*}
Using the boundedness of $V$, Lemma \ref{Zij}, and the dominated convergence theorem (since $|\hat{q}|\to\infty$ as $\varepsilon\to 0$),
\begin{align*}
\bigg|\int_{\R^n\setminus B_{R_0}(\hat{q})}(V(\varepsilon x)-V(\hat{\xi}))w_\ell Z_{ij}\,dx\bigg|&=O\bigg(\eta^{-n-2s}\int_{\R^n\setminus B_{R_0}(\hat{q})}|Z_{ij}|\,dx+\eta^{-\nu_1}\int_{\R^n\setminus B_{R_0}(\hat{q})}|w_\ell|\,dx\bigg)\\
&=o\big(\eta^{-n-2s}+\eta^{-\nu_1}\big).
\end{align*}
Likewise, using the fact that $|V(\xi^\varepsilon_\ell)-V(\hat{\xi})|\to 0$ for every $\ell\in\{1,\ldots,k\}$, it can be seen that 
$$\sum_{\ell\neq i}^k\int_{\Omega_\varepsilon}(V(\hat{\xi})-\lambda_\ell)w_\ell Z_{ij}\,dx=o\big(\eta^{-n-2s}+\eta^{-\nu_1}\big).$$
Thus, reproducing computation \eqref{compwiZij} we conclude
\begin{equation}\label{ec2}
\sum_{\ell=1}^k\int_{\Omega_\varepsilon}(V(\varepsilon x)-\lambda_\ell)\bar{u}_\ell Z_{ij}\,dx=\varepsilon c_0\frac{\partial V(\xi_i)}{\partial x_j}+o(\eta^{-n-2s})+O(\varepsilon^2), 
\end{equation}
with $c_0$ a negative constant.

On the other hand, using Lemma \ref{graZij}, \eqref{distqell}, and \cite[Lemma 3.1]{ARS2021}, 
\begin{equation}\begin{split}\label{ec3}
\int_{\Omega_\varepsilon}\bigg(\bigg(\sum_{\ell=1}^kw_\ell\bigg)^p-\sum_{\ell=1}^kw_\ell^p\bigg)Z_{ij}\,dx&=
\int_{\Omega_i}\bigg(\bigg(\sum_{\ell=1}^kw_\ell\bigg)^p-\sum_{\ell=1}^kw_\ell^p\bigg)Z_{ij}\,dx+o_\varepsilon\big(\eta^{-(n+2s)}\big)\\
&=p\int_{\Omega_i}w_i^{p-1}\sum_{\ell\neq i} w_\ell Z_{ij}\,dx+o_\varepsilon\big(\eta^{-(n+2s)}\big)
\\
&=p\int_{\R^n}w_i^{p-1}\sum_{\ell\neq i} w_\ell Z_{ij}\,dx+o_\varepsilon\big(\eta^{-(n+2s)}\big)+O(\varepsilon^{n+2s})\\
&=\sum\limits_{\ell\neq i}^k \frac{\gamma_\ell}{|q_\ell^\varepsilon-q_i^\varepsilon|^{n+2s}}\left(\frac{q_\ell^\varepsilon-q_i^\varepsilon}{|q_\ell^\varepsilon-q_i^\varepsilon|}\right)_j +o_\varepsilon\big(\eta^{-(n+2s)}\big)+O(\varepsilon^{n+2s}).
\end{split}\end{equation}
Here $o_\varepsilon(\cdot)$ denotes the standard little $o(\cdot)$ notation when we take the limit $\varepsilon\to 0$. Putting together \eqref{ec0}-\eqref{ec3} the result follows. 
\end{proof}

\textbf{Proof of Theorem \ref{result4}.}  
Since by hypothesis
$$\frac{|\xi^\varepsilon_i-\xi^\varepsilon_\ell|}{\varepsilon}\to +\infty\quad \mbox{ as }\varepsilon\to 0,$$
for every $i,\ell\in\{1,\ldots, k\}$, then, given $C\geq 0$,
$$2\max_{\ell\in\{1,\ldots, k\}}|\xi^\varepsilon_\ell-\hat{\xi}|\geq |\xi^\varepsilon_1-\xi^\varepsilon_2|\geq C\varepsilon,$$
provided $\varepsilon$ is small enough.
Then, from Lemma \ref{nonexistence2} we deduce 
\begin{equation}\label{final}-\varepsilon\frac{\partial V(\xi_i^\varepsilon)}{\partial x_j}+c\sum\limits_{\ell\neq i}\frac{1}{|q_\ell^\varepsilon-q_i^\varepsilon|^{n+2s}}\left(\frac{q_\ell^\varepsilon-q_i^\varepsilon}{|q_\ell^\varepsilon-q_i^\varepsilon|}\right)_j+O\bigg(\varepsilon \max_{\ell\in\{1,\ldots, k\}}|\xi^\varepsilon_\ell-\hat{\xi}|\bigg)+o(\eta^{-(n+2s)})=0
\end{equation}
for every $i =1,...,k, j =1,...,n,$ which is exactly the identity in \cite[Lemma 6.2]{KW2000}. 

The proof of the theorem thus follows exactly as in \cite[Theorem 1.2]{KW2000}, by means of \eqref{final} and a contradiction argument.
$\hfill\square$\\

\section*{Acknowledgements}
M.M. and J.W. are supported by Proyecto de Consolidaci\'on Investigadora, CNS2022-135640, MICINN (Spain).  M.M. is also supported by RYC2020-030410-I, MICINN (Spain) and by the grant PID2023-149451NA-I00 of MCIN/AEI/10.13039/ 501100011033/FEDER, UE. 
%% ---------------------------------------------------
%%% ---------------------------------------------------
%%% ---------------------------------------------------
%\section*{References}

\bibliography{schrodinger.bib}
\bibliographystyle{abbrv}

\end{document}